\pdfoutput=1

\documentclass[12pt,a4paper,reqno]{amsart}
\usepackage{microtype}
\usepackage[utf8]{inputenc}

\usepackage{todonotes}

\usepackage[style=ieee-alphabetic, url=false, backref=true, doi=false, isbn=false, eprint=false]{biblatex}

\usepackage{amsmath,amstext,amssymb,mathrsfs,amscd,amsthm,indentfirst, bbm}
\usepackage{amsfonts}
\usepackage{stmaryrd} 

\usepackage{float}
\usepackage{adjustbox}
\usepackage{booktabs}
\usepackage{verbatim}
\usepackage{enumerate}
\usepackage{graphicx}
\usepackage{textcomp}
\usepackage{tikz,tikz-cd}
\newcommand*\circlednum[1]{\tikz[baseline=(char.base)]{
            \node[shape=circle,draw,inner sep=2pt] (char) {#1};}}
\tikzset{circled/.style={draw,circle,font=\normalsize,inner sep=2pt}} 
\usetikzlibrary{patterns}
\usepackage{ifthen}
\usepackage{url}
\usepackage{hyperref}
\PassOptionsToPackage{hyphens}{url}\usepackage{hyperref}
\hypersetup{
    colorlinks  = true,
    linkcolor   = [rgb]{.20,.60,.22},
    citecolor   = [rgb]{0,.40,.80},
    urlcolor    = [rgb]{0,.40,.80}
}
\usepackage[nameinlink,capitalize]{cleveref}

\usepackage{libertine}
\makeatletter
\renewcommand\th@plain{\slshape}
\makeatother

\newcommand{\bC}{\mathbb C}

\newcommand{\bN}{\mathbb N}
\newcommand{\bP}{\mathbb P}
\newcommand{\bQ}{\mathbb Q}
\newcommand{\bR}{\mathbb R}

\newcommand{\bV}{\mathbb V}
\newcommand{\bZ}{\mathbb Z}

\newcommand{\cA}{\mathcal A}
\newcommand{\cB}{\mathcal B}
\newcommand{\cC}{\mathcal C}

\newcommand{\cE}{\mathcal E}
\newcommand{\cF}{\mathcal F}

\newcommand{\cI}{\mathcal I}

\newcommand{\cK}{\mathcal K}
\newcommand{\cL}{\mathcal L}
\newcommand{\cM}{\mathcal M}
\newcommand{\cN}{\mathcal N}
\newcommand{\cO}{\mathcal O}
\newcommand{\cP}{\mathcal P}
\newcommand{\cQ}{\mathcal Q}

\newcommand{\cU}{\mathcal U}

\newcommand{\cZ}{\mathcal Z}

\newcommand{\fm}{\mathfrak m}
\newcommand{\fM}{\mathfrak M}

\newcommand{\lra}{\longrightarrow}

\newcommand{\Gammans}{\Gamma_{\mathrm{ns}}}

\newcommand{\Gammacon}{\Gamma_{\cC^0}}

\newcommand{\Gammaconfib}{\Gamma_{\cC^0, \mathrm{fib}}}

\newcommand{\Mhyp}{\mathcal{M}_\mathrm{hyp}}
\newcommand{\Pic}{\mathrm{Pic}}

\newcommand{\Map}{\mathrm{Map}}

\newcommand{\NS}{\mathrm{NS}}
\newcommand{\Hilb}{\mathrm{Hilb}}
\newcommand{\Grass}{\mathrm{Gr}}
\newcommand{\UConf}{\mathrm{UConf}}

\newcommand{\BDiff}{B\mathrm{Diff}}

\DeclareMathOperator{\Spec}{Spec}
\DeclareMathOperator{\Proj}{Proj}
\DeclareMathOperator{\Sym}{Sym}

\newcommand{\catSch}{\mathsf{Sch}}
\newcommand{\catSet}{\mathsf{Set}}
\newcommand{\opp}{\mathrm{op}}

\newcommand{\PL}{\mathcal{P}_{[\mathcal{L}]}}

\newcommand{\stacks}[1]{\cite[\href{https://stacks.math.columbia.edu/tag/#1}{Tag #1}]{stacks-project}}

\newtheorem{theorem}{Theorem}[section]
\newtheorem*{theorem*}{Theorem}
\newtheorem{theorem-in-progress}[theorem]{Theorem-in-progress}

\newtheorem*{conjecture*}{Conjecture}

\newtheorem{definition}[theorem]{Definition}
\newtheorem{lemma}[theorem]{Lemma}
\newtheorem{remark}[theorem]{Remark}
\newtheorem{proposition}[theorem]{Proposition}
\newtheorem{example}[theorem]{Example}
\newtheorem{conjecture}[theorem]{Conjecture}
\newtheorem{corollary}[theorem]{Corollary}

\title{Scanning the moduli of smooth hypersurfaces}

\author{Alexis Aumonier}
\address{Center for Mathematical Sciences, Wilberforce Road, Cambridge CB3 0WB, UK}
\email{aa2099@cam.ac.uk}

\begin{document}

\maketitle

\begin{abstract}
We study the locus of smooth hypersurfaces inside the Hilbert scheme of a smooth projective complex variety. In the spirit of scanning, we construct a map to a continuous section space of a projective bundle, and show that it induces an isomorphism in integral homology in a range of degrees growing with the ampleness of the hypersurfaces. When the ambient variety is a curve, this recovers a result of McDuff about configuration spaces. We compute the rational cohomology of the section space and exhibit a phenomenon of homological stability for hypersurfaces with first Chern class going to infinity. For simply connected varieties, the rational cohomology is shown to agree with the stable cohomology of a moduli space of hypersurfaces, with a peculiar tangential structure, as studied by Galatius and Randal-Williams.
\end{abstract}

\section{Introduction}

A hypersurface, also known as an effective Cartier divisor, in a smooth projective complex variety $X$ is the zero locus $V(s)$ of a non-zero global section $s \in H^0(X, \cL)$ of an algebraic line bundle $\cL$:
\[
    V(s) := \left\{ x \in X \mid s(x) = 0 \right\} \subset X.
\]
Such a hypersurface is said to be smooth when the derivative $ds(x) \neq 0$ for all $x \in V(s)$. Smooth hypersurfaces obtained from sections of $\cL$ are parameterised by the complement of the discriminant inside the complete linear system $|\cL| = \bP\big( H^0(X, \cL) \big)$, a classical object dating back to the early days of complex geometry. To allow variations of the line bundle, Grothendieck introduced in \cite{grothendieck_technique_1962} the functor of relative effective Cartier divisors. We follow his footsteps and, for a given polynomial $P \in \bQ[x]$, consider the functor on the category of complex schemes
\begin{align*}
    \fM^{\mathrm{sm},P} \colon \catSch_\bC^\opp &\lra \catSet \\
    T &\longmapsto \left\{ Z \subset X \times T\ \middle\vert \begin{array}{l}
    Z \to T \text{ is flat and proper and for all } t \in T \\
    Z_t \subset X_t \text{ is a smooth effective Cartier divisor} \\
    \text{with Hilbert polynomial } \chi\big(X_t, \cI_{Z_t}^{-1}(n)\big) = P(n)
  \end{array}\right\}
\end{align*}
As already discovered by Grothendieck, this functor is representable by an open subscheme $\cM^{\mathrm{sm},P} \subset \Hilb(X)$ of the Hilbert scheme of $X$. In fact, under conditions on the polynomial $P$, it can be explicitly constructed using projective bundles \cite[Section~8.2]{bosch_neron_1990}, and it decomposes further into a disjoint union of connected components
\begin{equation}\label{equation:disjointunion}
    \cM^{\mathrm{sm},P} = \bigsqcup_\alpha \Mhyp^\alpha    
\end{equation}
indexed by certain cohomology classes $\alpha \in \NS(X) \subset H^2(X;\bZ)$ in the Néron--Severi group of $X$. More precisely, recording the isomorphism class of a Cartier divisor produces the \emph{Abel--Jacobi morphism}
\[
    \cM^{\mathrm{sm},P} \lra \Pic^P(X), \quad Z \longmapsto [\cI_Z^{-1}]
\]
to the Picard scheme, which parameterises isomorphism classes of line bundles on $X$ with Hilbert polynomial $P$. By the Hirzebruch--Riemann--Roch theorem, the first Chern class of $\cI_Z$ determines the polynomial $P(n) = \chi(X, \cI_Z^{-1}(n))$. We therefore have a finer decomposition of the Picard scheme into a finite\footnote{Given a polynomial $P$, the Hirzebruch--Riemann--Roch theorem imposes polynomial conditions on the possible first Chern classes $\alpha$ of hypersurfaces with Hilbert polynomial $P$. One observes that only finitely many $\alpha$ can satisfy the conditions.} disjoint union of tori
\[
    \Pic^P(X) = \bigsqcup_\alpha \Pic^\alpha(X),
\]
each parameterising line bundles with first Chern class $\alpha$. Its preimage under the Abel--Jacobi map is the disjoint union~\eqref{equation:disjointunion}, and our main objects of interest in this article are its connected components. Accordingly, we call such a component the \emph{moduli space of smooth hypersurfaces} with fixed Chern class
\[
    \Mhyp^\alpha \cong \left\{ Z \subset X \text{ smooth hypersurface with } c_1(\cO_X(Z)) = \alpha \right\}.
\]
It is a parameter space whose points are smooth complex hypersurfaces embedded inside $X$. In particular, the fibre of the Abel--Jacobi map $\Mhyp^\alpha \to \Pic^\alpha(X)$ above an isomorphism class of a line bundle $[\cL] \in \Pic^\alpha(X)$ is the complement of the discriminant inside the linear system $|\cL|$.

\bigskip

In the present work, we investigate the topology of $\Mhyp^\alpha$ using tools from algebraic topology. To state our results, we write $d(\alpha)$ for the largest integer $d$ such that all line bundles with first Chern class $\alpha$ are $d$-jet ample, and $\bP(J^1\cO_X)$ for the projectivisation of the first jet bundle of the structure sheaf $\cO_X$ (we review this notion in \cref{subsection:preliminaries-jet-bundle}).
\begin{theorem}[See \cref{theorem:main-theorem} for a precise version]
Let $X$ be a smooth projective complex variety and let $\alpha \in \NS(X)$ be ample enough. Taking the first jet yields a map
\[
    \Mhyp^\alpha \lra \Gammacon \big( \bP(J^1\cO_X) \big)
\]
which induces an isomorphism in integral homology onto the path component that it hits, in degrees $* < \frac{d(\alpha) - 3}{2}$.
\end{theorem}

\begin{remark}
In fact, $\pi_0\big(\Gammacon( \bP(J^1\cO_X))\big) \cong H^2(X;\bZ)$ and the jet map hits the component corresponding to $\alpha \in \NS(X) \subset H^2(X;\bZ)$.
\end{remark}

We recall in \cref{section:preliminaries} the notion of jet ampleness, and explain in \cref{section:range-estimates} how to estimate the number $d(\alpha)$. Among other things, we show that given any integer $M \geq 0$ and classes $\alpha, \beta \in \NS(X)$ with $\beta$ ample, we have that $d(\alpha + k\beta) \geq M$ for large enough $k \gg 0$. (See \cref{proposition:arbitrarily-big-jet-ampleness}.) In particular, the degree range of our main theorem can be arbitrarily big. In the remainder of this introduction, we describe applications of our main theorem and connect our results to the existing literature on moduli spaces of manifolds. 

\subsection{Rational computations and stability} A main advantage of our main theorem resides in the fact that the homotopy type of spaces of continuous sections can be approached by purely homotopical methods. This is particularly effective if one is willing to look at the rational information only. Using tools from rational homotopy theory, we show:
\begin{theorem}[See \cref{theorem:rational-cohomology-section-space-projective-bundle} for a precise version]
Let $n$ be the complex dimension of $X$. Let $\alpha \in H^2(X;\bZ)$ and denote by $\Gammacon^\alpha(\bP(J^1\cO_X))$ the component hit by the jet map from $\Mhyp^\alpha$. The rational cohomology ring of that section space is computed by the cohomology of the following commutative differential graded algebra:
\begin{gather*}
    \mathrm{Sym}^*_\mathrm{gr}\big( \bQ z \oplus H_1(X;\bQ) \oplus H_*(X;\bQ)[1] \big), \\
    \text{with } d(z) = 0, \ d(H_1(X;\bQ)) = 0, \text{ and } d(x) = \varphi(x) \text{ for } x \in H_*(X;\bQ)[1].
\end{gather*}
Here $\mathrm{Sym}^*_\mathrm{gr}$ denotes the free graded commutative algebra on a graded vector space, $[1]$ increases the grading by one, and $\bQ z$ is a one-dimensional vector space generated by $z$ in degree $2$. The differential is encoded by a morphism $\varphi \colon H_*(X;\bQ) \to \mathrm{Sym}^*_\mathrm{gr}( \bQ z \oplus H_1(X;\bQ))$ which can be computed explicitly in terms of the Chern classes of $J^1\cO_X$ and $\alpha$.
\end{theorem}

As to emphasise the words ``explicitly computable" in the statement above, we have worked out completely the formulas in \cref{example:torus-rational-cdga} in the case of $X$ being a torus. Let us also mention another application of our main theorem in the form of a rational homological stability phenomenon:
\begin{theorem}[See \cref{corollary:rational-stability}]
Let $X$ be a smooth projective complex variety whose tangent bundle is a topologically trivial complex vector bundle. Let $\alpha \in \NS(X)$ be ample and assume that $d(\alpha) \geq 1$. Then, for any integer $k \geq 1$, there is map
\[
    \Mhyp^{k\alpha} \lra \Map_{\alpha}(X,\bP^n_\bQ)
\]
inducing an isomorphism in rational homology in the range of degrees $* < \frac{k \cdot d(\alpha) - 3}{2}$. In particular, the rational homology stabilises as $k \to \infty$.
\end{theorem}

\begin{remark}
By a theorem of Wang \cite{wang_complex_1954}, if the tangent bundle of $X$ is holomorphically trivial then $X$ is an abelian variety. If we only require topological triviality as in the theorem above, other examples exist such as products of abelian varieties with curves \cite{schemer_what_2015}.
\end{remark}

\subsection{Configuration spaces on curves}
On an algebraic curve, the first Chern class of a line bundle is simply its degree under the identification $H^2(X;\bZ) \cong \bZ$. The vanishing locus of a non-zero global section of a line bundle of degree $d > 0$ is a set of $d$ points counted with multiplicity. In fact, such a section is non-singular precisely when these $d$ points are distinct. In that case, the moduli space of hypersurfaces is the classically studied configuration space and our main theorem recovers parts of a result of McDuff \cite[Theorem~1.1]{mcduff_configuration_1975}, who studied more generally configuration spaces on any smooth manifold via scanning methods:
\begin{theorem}[See \cref{theorem:scanning-for-curves}]
Let $X$ be a smooth projective complex curve of genus $g$, and denote by $\dot{TX}$ the fibrewise one-point compactification of the tangent bundle of $X$. Let $\alpha \in H^2(X;\bZ) \cong \bZ$ be such that $\alpha > 2g-2$. The jet map
\[
    \UConf_\alpha(X) \cong \Mhyp^\alpha \lra \Gammacon^\alpha(\bP(J^1\cO_X)) \cong \Gammacon^\alpha(\dot{TX})
\]
induces an isomorphism in integral homology in the range of degrees $* < \frac{\alpha - 2g - 3}{2}$.
\end{theorem}

\subsection{Characteristic classes and moduli spaces of manifolds}

Let $H = V(s) \subset X$ be a hypersurface defined by a non-singular section of a line bundle $\cL$ on $X$. In the series of papers \cite{galatius_stable_2014,galatius_homological_2017,galatius_homological_2018}, Galatius and Randal-Williams have investigated the homology of classifying spaces of diffeomorphisms groups of manifolds. In this article, we have tried to compare their results to ours in the case where the manifold under investigation is $H$. Deferring the technical details to \cref{section:comparison-with-diffeomorphisms}, we describe here the main contents.

Let $n$ be the complex dimension of $X$, so that $H$ is of real dimension $2n-2$. A tangential structure for $(2n-2)$-manifolds is a fibration $\theta \colon B \to BO(2n-2)$ which classifies a vector bundle $\theta^*\gamma_{2n-2}$ of real rank $2n-2$, obtained by pulling back the universal bundle $\gamma_{2n-2}$ on $BO(2n-2)$. A $\theta$-structure on $H$ is then defined as bundle map $\hat\ell \colon TH \to \theta^*\gamma_{2n-2}$. Informally, such a structure amounts to choosing a lift of the map $\tau_H$ classifying the tangent bundle of $H$
\begin{center}
\begin{tikzcd}
                                                  & B \arrow[d, "\theta"] \\
H \arrow[r, "\tau_H"'] \arrow[ru, "\ell", dashed] & BO(2n-2)             
\end{tikzcd}
\end{center}
where we have written $\ell$ for the map of spaces underlying $\hat\ell$. For instance, if $\theta \colon BSO(2n-2) \to BO(2n-2)$ is the natural map, a $\theta$-structure amounts to the choice of an orientation. Writing $\mathrm{Bun}^\theta(H)$ for the space of bundle maps $TH \to \theta^*\gamma_{2n-2}$, the moduli space classifying smooth $H$-bundles with $\theta$-structure is defined as the homotopy orbit construction
\[
    \cM^\theta(H) = \mathrm{Bun}^\theta(H) \sslash \mathrm{Diff}(H)
\]
and we denote by $\cM^\theta(H, \hat\ell) \subset \cM^\theta(H)$ the path component of $\hat\ell$.

One could wonder which tangential structure is the most natural on $H$ (eg. $H$ is at least canonically oriented as a complex manifold). In the last section of this article, we find a reasonable candidate and show:
\begin{theorem}[See \cref{theorem:comparison-of-characteristic-classes}, \cref{corollary:simply-connected-gammans-bdiff} and \cref{corollary:without-mu-quotient} for precise versions]
Let $X$ be a simply connected smooth complex projective variety of dimension $n \geq 4$ and $\cL$ be a very ample line bundle on it with $\alpha = c_1(\cL)$. There is a map $\theta \colon B \to BO(2n-2)$ depending solely on $\cL$ such that a hypersurface $H$ defined by a non-singular section of $\cL$ naturally inherits a $\theta$-structure $\hat\ell$. For this tangential structure, the map classifying the universal bundle
\[
    \Mhyp^\alpha \lra \cM^\theta(H,\hat\ell)
\]
induces an isomorphism in rational cohomology in the stable ranges.
\end{theorem}

\subsection{Outline} In \cref{section:preliminaries} we recall known properties about the Picard scheme, jet bundles, jet ampleness, and the topology of smooth hypersurfaces. In \cref{section:moduli-of-hypersurfaces} we define precisely the moduli space $\Mhyp^\alpha$. In \cref{section:statement-main-theorem,,section:proof-main-theorem} we state and prove our main theorem. The rest of the paper is dedicated to applications. We make rational computations in \cref{section:rational-homotopy} and describe the relation to scanning and characteristic classes of manifold bundles in \cref{section:scanning-configuration-spaces,,section:characteristic-classes}. Finally, we have assembled in \cref{section:range-estimates} various results concerning jet ampleness.

\subsection{The proof strategy for the main theorem} As explained in the introduction, we have a sequence of spaces
\begin{equation}\label{equation:kind-of-fibre-sequence}
    |\cL| \setminus \Delta \lra \Mhyp^\alpha \lra \Pic^\alpha(X)
\end{equation}
where $\Delta \subset |\cL|$ is the discriminant hypersurface. It turns out not to be a fibration, but only a microfibration. On the other hand, we have an actual fibration
\begin{equation}\label{equation:an-actual-fibre-sequence}
    \Gammacon(J^1\cL \setminus 0) / \bC^\times \lra \Gammacon^\alpha(\bP(J^1\cO_X)) \lra B\big(\Map(X,\bC^\times)/\bC^\times\big)
\end{equation}
obtained by modding out by the constant functions $\bC^\times \subset \Map(X,\bC^\times)$ and delooping the $\Map(X,\bC^\times)$-principal fibration
\[
    \Map(X,\bC^\times) \lra \Gammacon(J^1\cL \setminus 0) \lra \Gammacon^0(\bP(J^1\cL)) \cong \Gammacon^\alpha(\bP(J^1\cO_X))
\]
sending a section of $J^1\cL \setminus 0$ to its projectivisation. We observe the weak homotopy equivalences
\[
    B\big(\Map(X,\bC^\times)/\bC^\times\big) \simeq K(H^1(X;\bZ),1) \simeq \Pic^\alpha(X)
\]
and have proved in the earlier work \cite{aumonier_h-principle_2022} that the jet map
\[
    |\cL| \setminus \Delta \lra \Gammacon(J^1\cL \setminus 0) / \bC^\times
\]
induces an isomorphism in homology in a range of degrees. In essence, the proof consists in comparing with two (micro)fibrations~\eqref{equation:kind-of-fibre-sequence} and~\eqref{equation:an-actual-fibre-sequence}: we will leverage the homotopy (resp. homology) equivalence of their bases (resp. fibres) to obtain a homology equivalence between their total spaces.

\setcounter{tocdepth}{1}
\hypersetup{bookmarksdepth=3}
\tableofcontents

\subsection*{Acknowledgements} This work is part of my PhD thesis and I would like to thank my advisor Søren Galatius for many useful discussions, suggestions, and encouragements. Thanks to Ronno Das as well, for comments on this project and joint work on another where we heavily use microfibrations. I was supported by the Danish National Research Foundation through the Copenhagen Centre for Geometry and Topology (DNRF151) as well as the European Research Council (ERC) under the European Union's Horizon 2020 research and innovation programme (grant agreement No. 682922). 

\section{The Picard scheme, jet bundles, and hypersurfaces}\label{section:preliminaries}

\subsection{The Picard scheme}

We begin with recollections on line bundles on smooth projective complex varieties and their moduli. A more precise, and much more general, account of this standard material can be found in Kleiman's notes \cite{kleiman_picard_2005}. In what follows, $X$ is a connected smooth projective complex variety.

\begin{definition}
The \emph{absolute Picard group} of a complex scheme $T$ is the set $\Pic_\mathrm{abs}(T)$ of isomorphism classes of algebraic line bundles on $T$ equipped with the group law given by the tensor product.
\end{definition}

\begin{definition}\label{definition:picard-scheme}
The Picard functor of $X$
\[
    T \lra \Pic_\mathrm{abs}(X \times T) / \Pic_\mathrm{abs}(T)
\]
from complex schemes to abelian groups is represented by a scheme $\Pic(X)$ called the \emph{Picard scheme} of $X$ (relative to $\Spec(\bC)$).
\end{definition}

\begin{remark}
The analytification of $\Pic(X)$, sometimes called the \emph{Picard space} in this article, is the group of isomorphism classes of holomorphic line bundles on $X$.
\end{remark}

\begin{lemma}[{Compare \cite[Exercise 4.3]{kleiman_picard_2005}}]\label{lemma:existence-poincare-bundle}
There exists a (non-unique) algebraic line bundle $\cP$ on $\Pic(X) \times X$ satisfying the following property: given any complex scheme $T$ and a line bundle $\cL$ on $X \times T$, there exists a unique morphism $h \colon T \to \Pic(X)$ such that
\[
    \cL \cong (1 \times h)^*\cP \otimes f^*\cN
\]
for $f \colon X \times T \to T$ the projection map and $\cN$ some line bundle on $T$. \qed
\end{lemma}

\begin{definition}\label{definition:poincare-line-bundle}
Any choice of a line bundle $\cP$ as in \cref{lemma:existence-poincare-bundle} will be called a \emph{Poincaré line bundle} on $X$.
\end{definition}

The Picard scheme $\Pic(X)$ only parameterises \emph{isomorphism classes} of line bundles on $X$. One should think of the choice of a Poincaré line bundle as making compatible choices of \emph{representatives} of those isomorphisms classes.

We will need a further decomposition of the Picard scheme into components. To introduce it, we let $\cO_X(1)$ be a very ample line bundle on $X$ and write $\cF(n) = \cF \otimes \cO_X(1)^{\otimes n}$ for any sheaf of $\cO_X$-modules $\cF$ and $n \in \bZ$.
\begin{definition}\label{definition:hilbert-polynomial}
Let $\bC \subset k$ be a field extension, and write $X_k = X \times_{\Spec(\bC)} \Spec(k)$ for the base change. The \emph{Hilbert polynomial} of a closed subscheme $Z \subset X_k$ is the function
\[
    P_Z \colon \bN \lra \bZ, \quad n \longmapsto \chi(Z, \cO_Z(n))
\]
where $\cO_Z$ is the structure sheaf of $Z$.
\end{definition}
\noindent Given a polynomial $P \in \bQ[x]$, let $\Pic^P(X)(-) \subset \Pic(X)(-)$ be the subfunctor of the Picard functor whose $T$-points are represented by the line bundles $\cL$ on $X \times T$ such that
\[
    \chi(X_t, \cL_t^{-1}(n)) = P(n) \quad \text{ for all } t \in T
\]
where $X_t$ and $\cL_t$ denote the base change to $t$.

\begin{proposition}[Compare {\cite[Theorems~4.8~and~6.20]{kleiman_picard_2005}}]\label{proposition:picard-scheme-hilbert}
The Picard functor $\Pic^P(X)(-)$ is represented by a complex quasi-projective scheme denoted $\Pic^P(X)$. The Picard scheme is the disjoint union of the $\Pic^P(X)$ when $P$ runs over all polynomials $P$. \qed
\end{proposition}

Passing to complex points, the picture is vastly simplified by Hodge theory as recalled in the following two results.

\begin{definition}\label{definition:neron-severi-group}
The \emph{Néron--Severi} group of $X$, denoted $\NS(X)$, is the image of the morphism $c_1 \colon \Pic(X) \to H^2(X;\bZ)$ sending an isomorphism class of a line bundle to its first Chern class.
\end{definition}

\begin{proposition}\label{proposition:description-complex-picard}
The Picard space $\Pic(X)$ a disjoint union of connected components indexed by the Néron--Severi group
\[
    \Pic(X) = \bigsqcup_{\alpha \in \NS(X)} \Pic^\alpha(X).
\]
Each space $\Pic^\alpha(X)$ is a torus, non-canonically isomorphic to $H^1(X; \cO_X) / H^1(X; \bZ)$, and parameterises isomorphism classes of holomorphic line bundles on $X$ with Chern class $\alpha$. \qed
\end{proposition}

\subsection{Jet bundles}\label{subsection:preliminaries-jet-bundle}

We recall the definition of jet bundles in algebraic geometry and explain the construction of the jet evaluation map which will be used throughout this article. The main reference for this section is \cite[Section 16.7]{grothendieck_elements_1967}. In this section only, the full generality offered by schemes will be convenient, so we momentarily work in this setting.

Let $f \colon Y \to S$ be a morphism of schemes, $\Delta \colon Y \to Y \times_S Y$ be the diagonal and $\cI$ be its ideal sheaf. We let $p_i \colon Y \times_S Y \to Y$ be the two projections for $i=1,2$.
\begin{definition}\label{definition:relative-jet-bundle}
Let $\cF$ be an $\cO_Y$-module. Its \emph{relative $r$-jet bundle} is defined by
\[
    J^r_{Y/S} \cF := (p_1)_* \big( \cO_{Y \times_S Y}/\cI^{r+1} \otimes p_2^*\cF\big).
\]
The two projections $p_i$ give two morphisms of sheaves of rings $\cO_Y \to J^r_{Y/S} \cO_Y$. We choose the one given by $p_1$ to define an $\cO_Y$-module structure. The other one, induced by $p_2$, is denoted by 
\[
    d^r_{Y/S} \colon \cO_Y \to J^r_{Y/S} \cO_Y
\]
and is called the \emph{jet map}. In particular $J^r_{Y/S} \cF = J^r_{Y/S} \cO_Y \otimes_{\cO_Y} \cF$ where $J^r_{Y/S} \cO_Y$ is seen as a right $\cO_Y$-module via $d^r_{Y/S}$. We will also write $d^r_{Y/S} \colon \cF \to J^r_{Y/S} \cF$ for the tensor of the jet map with $\cF$.
\end{definition}

The fibre of the jet bundle at a closed point $y \in Y$ with maximal ideal sheaf $\fm$ is $(J^r_{Y/S}\cF)|_y \cong \cF/ \fm^{r+1}\cF$. Intuitively, the jet map should be thought of as taking the $r$-th Taylor expansion of a function. In particular, as the Leibniz rule for differentation shows, it is not a morphism of $\cO_X$-modules when $r > 0$. On the contrary, taking the derivative of a function commutes with multiplication by a constant. At the level of the relative jet bundles, the functions on $S$ play the role of the scalars, and this fact is expressed by the following:
\begin{lemma}\label{lemma:pushforward-jet-OSmodule}
The pushforward of the jet map
\[
    f_* d^r_{Y/S} \colon f_*\cF \lra f_*J^r_{Y/S} \cF
\]
is a morphism of $\cO_S$-modules.
\end{lemma}
\begin{proof}
The claim can be checked locally on an affine cover. We can thus assume that $f \colon \Spec B \to \Spec A$ is a morphism between affine schemes, $\cF = \widetilde{M}$ is a $B$-module, and $I$ is the kernel of the multiplication map $B \otimes_A B \to B$. Then $f_*J^r_{X/S} \cF$ corresponds to
\[
    (B \otimes_A B)/I^{r+1} \otimes_{B \otimes_A B} ((B \otimes_A B) \otimes_B M)
\]
and $f_*\cF$ is $M$, both seen as $A$-modules via $f$. In these coordinates, the jet map is
\[
    m \longmapsto (1\otimes 1) \otimes ( (1\otimes 1) \otimes m)
\]
which is visibly $A$-linear.
\end{proof}

\begin{definition}\label{definition:jet-evaluation-morphism}
Let $\cF$ be an $\cO_Y$-module. The \emph{fibrewise jet evaluation map} is the composition of the pushforward of the jet map followed by the counit:
\[
    f^*f_*\cF \lra f^*f_* J^r_{Y/S}\cF \lra J^r_{Y/S}\cF.
\]
\end{definition}

To explain the definition above, we assume that $Y$ and $S$ are complex varieties for the rest of this section. As we will alternate between two points of view on vector bundles (as sheaves or as spaces over the base), it will sometimes be convenient to be explicit about which viewpoint is adopted:
\begin{definition}\label{definition:geometric-vector-bundle}
Let $\cF$ be a vector bundle (i.e. a locally free sheaf of $\cO_Y$-modules of finite rank) on a complex variety $Y$. The total space of the associated \emph{geometric vector bundle} is
\[
    \bV(\cF) := \underline{\Spec}_Y \big( \Sym^* (\cF^\vee) \big)^\mathrm{an}
\]
where $(-)^\mathrm{an}$ denotes the analytification functor.
\end{definition}

Suppose now that $\cF$ is a vector bundle on $Y$ such that $f_*\cF$ is also a vector bundle on $S$. As sets, we have an identification
\[
    \bV(f_*\cF) = \{ (s, \sigma) \mid s \in S, \ \sigma \in H^0(Y_s; \cF|_{Y_s}) \}
\]
where $Y_s = f^{-1}(s) \subset Y$ is the fibre above $s$. In particular, when $S = \Spec \bC$ is a point, this is the space of global sections $H^0(Y;\cF)$. In general, it should be thought of as a space of fibrewise sections.

\begin{lemma}\label{lemma:counit-is-evaluation}
Under the above assumptions, the counit map $f^*f_* \cF \to \cF$ induces the evaluation map 
\begin{align*}
    \bV(f^*f_*\cF) \cong \bV(f_*\cF) \times_S Y &\lra \bV(\cF) \\
    ((s,\sigma), y) &\longmapsto \sigma(y)
\end{align*}
on geometric realisation.
\end{lemma}
\begin{proof}
Recall first that $f^*f_*\cF$ is the sheafification of $U \longmapsto \cF(f^{-1}f(U)) \otimes_{\cO_S(f(U))} \cO_Y(U)$. (When $f(U)$ is not open, we mean taking the colimit over all open subsets of $S$ containing it.) Chasing through the adjunction, the counit map is then seen to be the sheafification of the map
\begin{align*}
    \cF(f^{-1}f(U)) \otimes_{\cO_S(f(U))} \cO_Y(U) &\lra \cF(U) \\
    a \otimes r &\longmapsto r \cdot a|_U
\end{align*}
and the claim follows.
\end{proof}

Summarising the situation, we see that the fibrewise jet evaluation map
\[
    \bV(f_*\cF) \times_S Y \lra \bV(J^r_{Y/S}\cF)
\]
is given above a point $s \in S$ by
\begin{align*}
    H^0(Y_s;\cF|_{Y_s}) \times Y &\lra \bV(J^r \cF|_{Y_s}) \\
    (\sigma, y) &\longmapsto j^r\sigma(y) \in \cF|_{Y_s} \big/ \fm^{r+1}\cF|_{Y_s}
\end{align*}
where $\fm$ is the maximal ideal sheaf of $y \in Y_s$. In other words, it takes the Taylor expansion of $\sigma$ at $y$ up to order $r$.

\subsection{Jet ampleness}

Having now defined jet bundles, we state the crucial definition of jet ampleness of a line bundle on a smooth projective complex variety $X$.
\begin{definition}[Compare \cite{beltrametti_generation_1999}]
Let $k \geq 0$ be an integer. Let $x_1, \ldots, x_t$ be $t$ distinct points in $X$ and $(k_1,\ldots,k_t)$ be a $t$-uple of positive integers such that $\sum_i k_i = k+1$. Denote by $\cO_X$ the structure sheaf of $X$ and by $\fm_i$ the maximal ideal sheaf corresponding to $x_i$. We regard the tensor product $\otimes_{i=1}^t \fm_i^{k_i}$ as a subsheaf of $\cO_X$ under the multiplication map $\otimes_{i=1}^t \fm_i^{k_i} \to \cO_X$. We say that a line bundle $\cL$ is \emph{$k$-jet ample} if the evaluation map
\[
    \Gamma\left(\cL\right) \lra \Gamma\left(\cL \otimes \left(\cO_X / \otimes_{i=1}^t \fm_i^{k_i}\right)\right) \cong \bigoplus_{i=1}^t \Gamma\left(\cL \otimes \left(\cO_X / \fm_i^{k_i}\right)\right)
\]
is surjective for any $x_1,\ldots,x_t$ and $k_1,\ldots,k_t$ as above.
\end{definition}

\begin{example}
Being $0$-jet ample corresponds to being spanned by global sections. Furthemore, $1$-jet ampleness is the usual notion of very ampleness. On a curve, a line bundle is $k$-jet ample whenever it is $k$-very ample. However, on higher dimensional varieties, a $k$-jet ample line bundle is also $k$-very ample but the converse is not true in general.
\end{example}

The following proposition is the main tool to produce line bundles having a very high degree of jet ampleness.
\begin{proposition}[{See~\cite[Proposition 2.3]{beltrametti_generation_1999}}]
If $\cA$ and $\cB$ are respectively $a$- and $b$-jet ample line bundles, then their tensor product $\cA \otimes \cB$ is $(a+b)$-jet ample.
\end{proposition}

\begin{definition}\label{definition:constant-minimal-jet-ampleness}
Let $X$ be a smooth projective complex variety and $\alpha \in H^2(X;\bZ)$. We write $d(X, \alpha)$ for the largest integer $d \geq -1$ such that all line bundles on $X$ with first Chern class equal to $\alpha$ are $d$-jet ample. (By convention, we declare that being $(-1)$-jet ample is an empty condition.)
\end{definition}

We refer to \cref{section:range-estimates} for how to compute $d(X, \alpha)$ in some special cases. Given an integer $d$, we also explain in \cref{proposition:arbitrarily-big-jet-ampleness} how to find an $\alpha$ such that $d(X, \alpha) \geq d$.

\subsection{The topology of hypersurfaces}

It is well known that all smooth degree $d$ complex hypersurfaces in $\bP^n$ are diffeomorphic. As a way of justifying the study of the moduli space of hypersurfaces of a given Chern class, we observe that such hypersurfaces are also all diffeomorphic, provided the Chern class is ample enough. First, recall that ampleness is a numerical property:
\begin{theorem}[Nakai--Moishezon criterion]
A line bundle $\cL$ on a proper scheme over a field is ample if and only if $\int_Y c_1(\cL)^{\dim Y} > 0$ for every integral subscheme $Y \subset X$.
\end{theorem}
\begin{definition}
A Chern class $\alpha \in \NS(X)$ is called \emph{ample} if it satisfies the Nakai--Moishezon criterion.
\end{definition}

We recall the following classical definition which is central in our work:
\begin{definition}
A global section $s \in \Gamma(X,\cL)$ of a line bundle $\cL$ on a smooth projective complex variety $X$ is called \emph{non-singular} if for all $x \in X$ we have $(s(x), ds(x)) \neq 0$.
\end{definition}

\begin{remark}
Given a non-singular section $s \in \Gamma(X,\cL)$, its vanishing locus
\[
    V(s) := \{ x \in X \mid s(x) = 0 \} \subset X
\]
is a smooth hypersurface.
\end{remark}

Any hypersurface $H$ can be seen as a Weil divisor, hence a Cartier divisor ($X$ is smooth), and therefore has an attached line bundle $\cO_X(H)$. If $H = V(s)$ with $s \in \Gamma(X,\cL)$, then $\cO_X(H) \cong \cL$. The following bit of language will be convenient:
\begin{definition}
The \emph{Chern class of a hypersurface} $H$ is the first Chern class of its associated line bundle $c_1(\cO_X(H))$.
\end{definition}

\begin{proposition}\label{proposition:diffeomorphic-hypersurfaces}
Let $X$ be a smooth projective complex variety with canonical sheaf $K_X$. Let $\alpha \in \NS(X)$ be a Chern class ample enough such that:
\begin{enumerate}
    \item the class $\alpha - c_1(K_X)$ is ample;
    \item for any line bundle $\cL$ of Chern class $\alpha$, the subspace $\Gammans(\cL) \subset \Gamma(\cL)$ consisting of the non-singular global sections is non empty.
\end{enumerate}
Then all the smooth hypersurfaces of Chern class $\alpha$ are diffeomorphic to one another.
\end{proposition}

\begin{remark}\label{remark:arrange-ampleness}
Let us make some remarks on the two assumptions of the proposition above. Let $\cL$ be a very ample line bundle on $X$. Then $K_X^{-1} \otimes \cL^{\otimes k}$ is very ample for $k \gg 0$ big enough, and $\alpha = c_1(\cL^{\otimes k})$ satisfies the first assumption. Furthermore, by Bertini theorem, the subspace $\Gammans(\cL) \subset \Gamma(\cL)$ is dense. The second assumption is thus satisfied as soon as all line bundles of Chern class $\alpha$ are very ample. We explain how to arrange this in \cref{section:range-estimates}.
\end{remark}

\begin{proof}
We first briefly recall the classical proof in the case of a single linear system. Let $\cL$ be a line bundle on $X$ with Chern class $\alpha$ and denote by $\Gammans(\cL) \subset \Gamma(\cL)$ the subset of those global sections that are non-singular. The projection from the incidence variety
\[
    \{ (s,x) \in \Gammans(\cL) \times X \mid s(x)=0 \} \lra \Gammans(\cL).
\]
is a proper surjective submersion between smooth manifolds, with fibres the smooth hypersurfaces. By Ehresmann's lemma, it is a fibre bundle and therefore all the fibres over a connected component are diffeomorphic. Finally, $\Gammans(\cL) \subset \Gamma(\cL)$ is the complement of the discriminant which has codimension at least 1, hence it is connected. 

\bigskip

Now we adapt the proof in families. Let $\Pic^\alpha(X)$ be the connected component of the Picard space classifying isomorphism classes of line bundles of Chern class $\alpha$, and let $\cP$ be a Poincaré line bundle on $\Pic^\alpha(X) \times X$. For $[\cL] \in \Pic^\alpha(X)$, we write $\cP_{[\cL]}$ for the line bundle on $X$ which represents the isomorphism class $[\cL]$. Let $p \colon \Pic^\alpha(X) \times X \to \Pic^\alpha(X)$ be the projection. By cohomology and base change \cite[Theorem~III.12.11]{hartshorne_algebraic_1977}, the sheaf $p_*\cP$ is a vector bundle provided that $H^1(X, \cP_{[\cL]}) = 0$ for all $[\cL] \in \Pic^\alpha(X)$. This follows by the Kodaira vanishing theorem and the assumption that $\alpha - c_1(K_X)$ is ample. Let
\[
    \bV(p_*\cP)^\mathrm{ns} \subset \bV(p_*\cP) = \{ ([\cL], s) \mid [\cL] \in \Pic^\alpha(X), \ s \in \Gamma(\cP_{[\cL]}) \}
\]
be the subset of those sections that are non-singular. The incidence variety
\[
    \{ ([\cL],s,x) \in \bV(p_*\cP)^\mathrm{ns} \times X \mid s(x) = 0 \} \lra \bV(p_*\cP)^\mathrm{ns}
\]
is a smooth fibre bundle by Ehresmann's lemma. The base is connected by our second assumption on the ampleness of $\alpha$, therefore all the fibres are diffeomorphic.
\end{proof}

\section{The moduli of hypersurfaces}\label{section:moduli-of-hypersurfaces}

In this section, we precisely define our main object of interest in this paper: the moduli of smooth hypersurfaces. From now on, we adopt the following conventions:
\begin{itemize}
    \item $X$ is a connected smooth complex projective variety;
    \item $\Pic(X)$ is its associated Picard scheme or space (see \cref{definition:picard-scheme});
    \item $p \colon \Pic(X) \times X \to \Pic(X)$ is the first projection;
    \item $\cP$ is a choice, once and for all, of a Poincaré line bundle (see \cref{definition:poincare-line-bundle});
    \item if $\cL$ is a line bundle on $X$, we write $\Gammans(\cL) \subset \Gamma(\cL)$ for the subspace of non-singular sections.
\end{itemize}

\subsection{Moduli functors and the Hilbert scheme}

In this subsection, we define the moduli functor of smooth hypersurfaces in $X$ and show that it is representable by an open subscheme of the Hilbert scheme of $X$. We give an explicit description of the representing moduli scheme using the theory of linear systems of divisors. Although providing motivation and context, this subsection is logically independent of the rest of the article. We must also say that the results presented here are well known to algebraic geometers, but we have chosen to recall them in detail to be self-contained. The reader unfamiliar with the algebro-geometric language used here is invited to jump to the \hyperref[subsection:point-set-model]{following subsection} where we provide a point-set model for the analytification of the moduli scheme, which will be thereafter used throughout the article. 

\bigskip

Given a polynomial $P \in \bQ[x]$, we may consider the Hilbert functor $\Hilb^P(X)(-)$ parameterising flat proper families of closed subschemes in $X$ with given Hilbert polynomial $P$ (recall \cref{definition:hilbert-polynomial}). In other words
\begin{align*}
    \Hilb^P(X)(-) \colon \catSch^\opp_\bC &\lra \catSet \\
    T & \longmapsto \left\{ Z \subset X \times T\ \middle\vert \begin{array}{l}
    Z \to T \text{ is flat and proper} \\
    \text{and } \forall t \in T, \ P_{Z_t}(n) = P(n)
  \end{array}\right\}
\end{align*}
where $\catSch_\bC$ is the category of schemes over $\Spec(\bC)$ and $Z_t \subset X_t$ is the fibre\footnote{If $k$ is the residue field of $t \in T$, then $Z_t = Z \times_T \Spec(k) \subset X_t = X_k$.} of $Z \to T$ above $t\in T$. More generally, the Hilbert functor $\Hilb(X)(-)$ of $X$ is the disjoint union of the $\Hilb^P(X)(-)$ as $P$ runs over all polynomials.
\begin{theorem}[Grothendieck {\cite{grothendieck_techniques_1961}}]
The Hilbert functor $\Hilb^P(X)(-)$ is represented by the \emph{Hilbert scheme} $\Hilb^P(X)$ which is projective over $\Spec(\bC)$.
\end{theorem}

Families of hypersurfaces are more commonly known as \emph{relative effective Cartier divisors}. (See \stacks{056P}.) Following \cite[Part 3]{kleiman_picard_2005}, we recall the definition of their moduli functor:

\begin{definition}
Let $P \in \bQ[x]$ be a polynomial. The \emph{moduli functor of effective divisors} with Hilbert polynomial $P$ is the functor
\begin{align*}
    \fM^P \colon \catSch_\bC^\opp &\lra \catSet \\
    T &\longmapsto \left\{ Z \subset X \times T\ \middle\vert \begin{array}{l}
    Z \to T \text{ is flat and proper and for all } t \in T \\
    \text{the ideal sheaf  } \cI_{Z_t} \text{ is a line bundle and is } \\
    \text{such that } \chi(X_t, \cI_{Z_t}^{-1}(n)) = P(n)
  \end{array}\right\}
\end{align*}
We define $\fM^{\mathrm{sm},P} \subset \fM^P$ to be the subfunctor\footnote{Recall that smoothness is preserved under base change.} where $Z_t$ is furthermore required to be smooth over the residue field $\Spec(\kappa(t))$ at $t \in T$.
\end{definition}

\begin{remark}
The moduli functor $\fM^P$ is visibly a subfunctor of the Hilbert functor. Though, because of our conventions, the indexing Hilbert polynomials are different. If $Z \subset X$ is an effective Cartier divisor and $P(n) = \chi(X, \cI_Z^{-1}(n))$, we let $P'$ be the associated polynomial $P'(n) = \chi(Z, \cO_Z(n))$. The subfunctor inclusion then reads $\fM^P \subset \Hilb^{P'}(X)$. 
\end{remark}

We are here working with a general projective variety $X$ with non-necessarily discrete Picard space, which can be confusing. To counteract that feeling, we remind the reader of the more classical, and enlightening, situation of linear systems of divisors:
\begin{example}[Compare {\cite[Definition~3.12 and Theorem~3.13]{kleiman_picard_2005}}]\label{example:moduli-linear-system}
Let $X = \bP^n$ and let $P' \in \bQ[x]$ be the Hilbert polynomial of a hypersurface of degree $d \geq 1$. One can show that the Hilbert scheme is in this case the complete linear system
\[
    \Hilb^{P'}(\bP^n) = |\cO(d)| = \bP\big(H^0(\bP^n,\cO(d))\big).
\]
Therefore $\Hilb^{P'}(\bP^n)(-) = \fM^P(-)$ is represented by a complex projective space of complex dimension $\dim_\bC H^0(\bP^n,\cO(d)) -1$. The subfunctor $\fM^{\mathrm{sm},P}$ is represented by the complement of the discriminant hypersurface.
\end{example}

In general, we have the following general representability result:
\begin{theorem}
The subfunctor $\fM^P(-) \subset \Hilb^{P'}(X)(-)$ is represented by a union of connected components of the Hilbert scheme $\Hilb^{P'}(X)$.
\end{theorem}
\begin{proof}
By \cite[Theorem~1.13]{kollar_rational_1996}, $\fM^P(-)$ is represented by an open subscheme $U$ of the Hilbert scheme such that the inclusion $U \subset \Hilb(X)$ is universally closed (this uses that $X$ is smooth over $\Spec(\bC)$). As the Hilbert scheme is separated (because projective), $U$ must be a union of connected components.
\end{proof}

Our goal in now to give an explicit description of the scheme representing $\fM^P(-)$. For the remainder of this section, we fix a polynomial $P \in \bQ[x]$, denote by $\Pic^P(X) \subset \Pic(X)$ the subscheme defined in \cref{proposition:picard-scheme-hilbert}, and still write $\cP$ for the restriction of the chosen Poincaré line bundle to it. Recall that $p \colon \Pic^P(X) \times X \to \Pic^P(X)$ denotes the first projection. Let
\[
    \cK := \ker\big( p^*p_* \cP \to \cP \big)
\]
be the kernel sheaf of the counit of $p^* \dashv p_*$. From the monoidality of $p^*$ and standard properties of the relative Proj construction, we have an isomorphism
\[
    \underline{\Proj}_{\Pic^P(X) \times X}\big(\Sym p^*(p_*\cP)^\vee\big) \cong \underline{\Proj}_{\Pic^P(X)}\big(\Sym (p_*\cP)^\vee\big) \times X.
\]
Using the surjection of sheaves $p^*(p_*\cP)^\vee \twoheadrightarrow \cK^\vee$ we thus obtain a closed immersion
\[
    \underline{\Proj}_{\Pic^P(X) \times X}\big(\Sym \cK^\vee\big) \subset \underline{\Proj}_{\Pic^P(X)}\big(\Sym (p_*\cP)^\vee\big) \times X.
\]
\begin{definition}
We define the \emph{universal family} to be the morphism
\[
    \cU := \underline{\Proj}_{\Pic^P(X) \times X}\big(\Sym \cK^\vee\big) \lra \underline{\Proj}_{\Pic^P(X)}\big(\Sym (p_*\cP)^\vee\big) =: \cM
\]
obtained by projecting onto the first coordinate.
\end{definition}

The following is the main theorem of this section:

\begin{theorem}\label{theorem:representability-moduli}
Assume that $H^1(X_t,\cP_t) = 0$ for all $t \in \Pic^P(X)$. Then the universal family $\cU \to \cM$ represents the moduli functor $\fM^P(-)$.
\end{theorem}

\begin{proof}
This is explained in \cite[Proposition~8.2.7]{bosch_neron_1990}. To translate to the notation in that book: take $f$ to be the structure morphism $X \to \Spec(\bC)$, $T$ to be $\Pic^P(X)$, and $\cL$ to be $\cP$. The flatness assumption on $\cL$ is implies by our assumption using cohomology and base change \cite[Theorem~III.12.11]{hartshorne_algebraic_1977}.
\end{proof}

\begin{remark}\label{remark:analytification-incidence-variety}
The analytification of the universal family is simply the incidence variety
\begin{center}
\begin{tikzcd}
{\big\{ (x, [\cL], [s]) \text{ with } x \in X, \ [\cL] \in \Pic^P(X), \ [s] \in \bP\Gamma(\cP_{[\cL]}) \text{ such that } s(x) = 0 \big\}} \arrow[d] \\
{\big\{ ([\cL], [s]) \text{ with } [\cL] \in \Pic^P(X), \ [s] \in \bP\Gamma(\cP_{[\cL]}) \big\}}                                                    
\end{tikzcd}
\end{center}
above the projectivisation of the vector bundle $\bV(p_*\cP) \to \Pic^P(X)$.
\end{remark}

\begin{definition}\label{definition:abel-jacobi-morphism}
The morphism of functors
\[
    \fM^P(T) \lra \Pic^P(X)(T), \quad (Z \subset X \times T) \longmapsto [\cI_Z^{-1}]
\]
is represented by the projection morphism
\[
    \cM = \underline{\Proj}_{\Pic^P(X)}\big(\Sym (p_*\cP)^\vee\big) \lra \Pic^P(X)
\]
which is usually called the \emph{Abel--Jacobi} morphism.
\end{definition}

We now explain how to obtain a scheme representing $\fM^{\mathrm{sm},P}(-)$. We assume that the evaluation morphism $p^*p_* \cP \to \cP$ is surjective. Recall the jet evaluation morphism 
\[
    p^*p_*\cP \lra J^1_{\Pic^P(X) \times X/ \Pic^P(X)} \cP
\]
from \cref{definition:jet-evaluation-morphism}. We have a commutative diagram with middle row a short exact sequence:
\begin{center}
\begin{tikzcd}
            &                                                & \big( J^1_{\Pic^P(X) \times X/ \Pic^P(X)} \cP \big)^\vee \arrow[d] &                                                  &   \\
0 \arrow[r] & \cP^\vee \arrow[r] \arrow[ru] \arrow[rd, "0"'] & p^*(p_*\cP)^\vee \arrow[r] \arrow[d]                               & \cK^\vee \arrow[r] \arrow[ld, two heads, dashed] & 0 \\
            &                                                & \cQ \arrow[d]                                                      &                                                  &   \\
            &                                                & 0                                                                  &                                                  &  
\end{tikzcd}
\end{center}
where $\cQ$ is defined to be the cokernel of the dual of the jet evaluation, and
\[
    \cP^\vee \lra \big( J^1_{\Pic^P(X) \times X/ \Pic^P(X)} \cP \big)^\vee
\]
is the dual of the projection morphism from the first jet bundle to the zeroth jet bundle $J^0\cP = \cP$. The composition $\cP^\vee \to p^*(p_*\cP)^\vee \to \cQ$ is the zero morphism, and we thus obtain a surjective morphism $\cK^\vee \twoheadrightarrow \cQ$.

\begin{corollary}\label{corollary:moduli-scheme-smooth-hypersurfaces}
Assume that the counit morphism $p^*p_* \cP \to \cP$ is surjective, and that $H^1(X_t,\cP_t) = 0$ for all $t \in \Pic^P(X)$. Then the moduli functor $\fM^{\mathrm{sm},P}(-)$ is represented by an open subscheme of $\cM$, hence of the Hilbert scheme of $X$.  More precisely, let $\pi \colon \cU \to \cM$ be the universal family and
\[
    \cZ := \underline{\Proj}_{\Pic^P(X) \times X}\big(\Sym \cQ\big) \hookrightarrow \cU
\]
be the closed subscheme determined by the surjection $\cK^\vee \twoheadrightarrow \cQ$. Then $\fM^{\mathrm{sm},P}(-)$ is represented by $\cM \setminus \pi(\cZ)$.
\end{corollary}
\begin{remark}
The sheaf $\cQ$ is only a coherent sheaf of $\cO_{\Pic^P(X) \times X}$-modules and $\cZ$ is therefore not a vector bundle in general. Nonetheless, if we furthermore assume that the jet evaluation morphism is surjective (e.g. if all line bundles parameterised by $\Pic^P(X)$ are very ample), then $\cQ$ is the dual of the kernel of a surjective morphism of locally free sheaves, hence itself locally free.
\end{remark}
\begin{remark}
In the notation of \cref{remark:analytification-incidence-variety}, $\cZ^\mathrm{an}$ consists of those points $(x, [\cL], [s])$ such that $j^1(s)(x) = 0$, i.e. $s$ is singular at $x$.
\end{remark}
\begin{proof}
Smoothness is an open condition: the projection $\pi$ is flat and of finite presentation, so \stacks{01V9} applies and $\fM^{\mathrm{sm},P}(-)$ is seen to be represented by an open subscheme of $\cM$. The subscheme $\cM \setminus \pi(\cZ) \subset \cM$ is open because $\pi$ is proper. It represents the moduli functor as smoothness can be checked locally using the Jacobian criterion.
\end{proof}

We close this section with some general remarks about our assumptions in \cref{corollary:moduli-scheme-smooth-hypersurfaces}. We show that, although stated scheme-theoretically, they can be checked after analytification.
\begin{lemma}
Let $\cP^\mathrm{an}$ denote the analytic sheaf associated to $\cP$. If the counit morphism $p^*p_* \cP^\mathrm{an} \to \cP^\mathrm{an}$ is surjective, then the same is true before analytification.
\end{lemma}
\begin{proof}
This follows from exactness of the analytification of sheaves functor.
\end{proof}

\begin{lemma}\label{lemma:vanishing-h1-complex-lines}
Let $P \in \bQ[x]$ be a polynomial and suppose that $H^1(X, \cL) = 0$ for all $[\cL] \in \Pic^P(X)(\bC)$. Then $H^1(X_t, \cP_t) = 0$ for all $t \in \Pic^P(X)$.
\end{lemma}
\begin{proof}
By upper semicontinuity of cohomology, the subscheme
\[
    \left\{ t \in \Pic^P(X) \mid H^1(X_t, \cP_t) = 0 \right\} \subset \Pic^P(X)
\]
is open. Assume that its complement is non empty. Then it contains a complex point: it is of locally of finite type over $\Spec(\bC)$ and Hilbert's Nullstellensatz applies. This cannot be the case by assumption.
\end{proof}

We finally comment on the relation between Hilbert polynomials and Chern classes. For a holomorphic line bundle $\cL$ on $X$, recall the Hirzebruch--Riemann--Roch theorem giving an equality
\[
    \chi(X, \cL) = \int_X \mathrm{ch}(\cL) \ \mathrm{td}(X)
\]
where $\mathrm{ch}(-)$ is the Chern character and $\mathrm{td}(X)$ is the Todd class of $X$. In particular, if $Z \subset X$ is an effective Cartier divisor, its Hilbert polynomial only depends on the first Chern class of its associated line bundle $\cO_X(Z)$. We thus obtain a numerical criterion:
\begin{lemma}
Let $P \in \bQ[x]$ and let $C$ be the collection of Chern classes
\[
    C := \left\{ c_1(\cL) \mid [\cL] \in \Pic^P(X)(\bC) \right\}.
\]
If $\alpha - c_1(K_X)$ is ample for all $\alpha \in C$, then $H^1(X_t, \cP_t) = 0$ for all $t \in \Pic^P(X)$.
\end{lemma}
\begin{proof}
This follows from \cref{lemma:vanishing-h1-complex-lines} whose assumption is verified by the Kodaira vanishing theorem.
\end{proof}

\subsection{A convenient point-set model}\label{subsection:point-set-model}

In this section, we unravel the result of \cref{corollary:moduli-scheme-smooth-hypersurfaces} and give an explicit point set model for the moduli space of smooth hypersurfaces.

\bigskip

We begin with notations which we will use throughout the rest of the article. If $[\cL] \in \Pic(X)$ is an isomorphism class of a line bundle, we write $\PL$ for the representative of that isomorphism class given by the restriction of $\cP$ to $X\cong \{[\cL]\} \times X \subset \Pic(X) \times X$. If $\alpha \in \NS(X)$ we recall from \cref{proposition:description-complex-picard} that $\Pic^\alpha(X) \subset \Pic(X)$ denotes the connected component parameterising line bundles of Chern class $\alpha$, and we will write $\cP_\alpha$ for the restriction of the Poincaré line bundle to that component.

\begin{definition}
The first jet bundle of $\cP$ relative to the projection $p$ (see \cref{definition:relative-jet-bundle}) is denoted
\[
    J^1_p\cP := J^1_{\Pic(X) \times X / \Pic(X)} \cP.
\]
When restricted to $\Pic^\alpha(X)$ for some $\alpha \in \NS(X)$, we will write $J^1_p\cP_\alpha := J^1_{\Pic^\alpha(X) \times X / \Pic^\alpha(X)}\cP_\alpha$.
\end{definition}

\begin{lemma}
Let $K_X$ be the canonical sheaf of $X$. Let $\alpha \in \NS(X)$ be such that $\alpha - c_1(K_X)$ is ample. Then $p_* \cP_\alpha$ is a vector bundle and the fibrewise jet evaluation map gives a map of vector bundles
\[
    p^*p_* \cP_\alpha \lra J^1_p\cP_\alpha.
\]
As sets, the geometric realisations are given by
\[
    \bV(p^*p_* \cP_\alpha) = \bV(p_*\cP_\alpha) \times X = \left\{ ([\cL], x, s) \mid [\cL] \in \Pic^\alpha(X), \ x \in X, \ s \in \Gamma(\PL) \right\}
\]
and
\[
    \bV(J^1_p\cP_\alpha) = \left\{ ([\cL], x, v) \mid [\cL] \in \Pic^\alpha(X), \ x \in X, \ v \in J^1\PL|_x \right\}.
\]
Under these identifications, the jet evaluation map is given by
\[
    \mathrm{jev} \colon ([\cL], x, s) \longmapsto ([\cL], x, j^1s(x)).
\]
\end{lemma}
\begin{proof}
The fact that $p_*\cP_\alpha$ is a vector bundle follows directly from cohomology and base change and the Kodaira vanishing theorem under the assumption that $\alpha - c_1(K_X)$ is ample. The rest of the lemma follows from the results recalled in \cref{subsection:preliminaries-jet-bundle}.
\end{proof}

Recall from \cref{corollary:moduli-scheme-smooth-hypersurfaces} the scheme $\cM \setminus \pi(\cZ)$ representing the moduli functor of smooth hypersurfaces. After analytification, we may restrict the Abel--Jacobi map
\begin{equation}\label{equation:abel-map-analytification}
    \big( \cM \setminus \pi(\cZ) \big)^\mathrm{an} \lra \Pic(X)^\mathrm{an}  
\end{equation}
to the connected component $\Pic^\alpha(X) \subset \Pic(X)$ (we will from now on drop the superscript ``$\mathrm{an}$"), recalled in \cref{proposition:description-complex-picard}, provided that $\alpha$ is ample enough:

\begin{definition}\label{definition:moduli-of-hypersurfaces}
Let $\alpha \in \NS(X)$ be such that $\alpha - c_1(K_X)$ is ample. The \emph{moduli of smooth hypersurfaces in $X$ of Chern class $\alpha$} is defined to be the preimage of $\Pic^\alpha(X)$ under the Abel--Jacobi map~\eqref{equation:abel-map-analytification}. In other terms, we have a homeomorphism:
\[
    \Mhyp^\alpha \cong \big( \bV(p_* \cP_\alpha) \setminus \mathrm{proj}(\mathrm{jev}^{-1}(0)) \big) / \bC^\times
\]
where the scalars act fibrewise over $\Pic(X)$, and $\mathrm{proj} \colon \bV(p^*p_*\cP) \to \bV(p_*\cP)$ is the projection induced by $p$.
\end{definition}

\begin{remark}\label{remark:point-set-description-moduli}
As sets, we have an identification
\[
    \Mhyp^\alpha = \{ ([\cL], [s]) \mid [\cL] \in \Pic^\alpha(X), \ [s] \in \Gammans(\cP_{[\cL]})/\bC^\times\}.
\]
That is, it will be technically convenient to think of a smooth hypersurface as a tuple consisting of a line bundle $\cL$ and a non-singular global section of it (up to isomorphism and scaling action). However, the name of moduli space is justified by the previous section: we have a homeomorphism
\[
    \Mhyp^\alpha \cong \left\{ Z \subset X \text{ smooth hypersurface with } c_1(\cO_X(Z)) = \alpha \right\} \subset \Hilb(X)^\mathrm{an}.
\]
\end{remark}

\section{Statement of the main theorem}\label{section:statement-main-theorem}

In this section, we construct a topological counterpart to the moduli of smooth hypersurfaces described in \cref{definition:moduli-of-hypersurfaces}. We then state our main theorem comparing the two objects.

\subsection{A topological counterpart}

We begin with some generalities about the topology of continuous section spaces. Let $E \to A \times B$ be a fibre bundle on a topological space $A \times B$. We denote a point of $E$ as a tuple $(a,b,e)$ where $a \in A$, $b \in B$ and $e \in E|_{(a,b)}$ is in the fibre. All mapping spaces are given the compact open topology.

\begin{definition}
The space of \emph{fibrewise sections} of $E$ over $A$ is defined to be the subspace
\begin{align*}
    \Gammaconfib(E \to A) := \left\{ (a, s) \mid a \in A, \ s \in \Gammacon(E|_{a \times B}) \right\} &\hookrightarrow \Map(B, E) \\
    (a,s) &\mapsto [b \mapsto (a,b,s(b))].
\end{align*}
\end{definition}

Post-composition with the projection maps $E \to A \times B \to A$ gives a continuous map $\Map(B, E) \to \Map(B ,A)$ which, when restricted to fibrewise sections, yields the projection map
\[
    \Gammaconfib(E \to A) \lra A, \quad (a,s) \longmapsto a.
\]
In particular, this projection map is continuous.

\begin{remark}
Let $Z$ be a topological space. A continuous map $Z \to \Gammaconfib(E \to A)$ is the same datum as a continuous map $f \colon Z \times B \to E$ over $B$ such that $\mathrm{proj} \circ f(z,-) \colon B \to E \to A$ is constant for any $z \in Z$.
\end{remark}

\begin{remark}\label{remark:holomorphic-continuous-fibrewise-sections}
When $A$ and $B$ are smooth projective complex varieties, one can modify the definition above by using the spaces of holomorphic maps instead of the whole mapping spaces. In fact, assume that $E = \bV(\cE)$ is also a vector bundle and that $\pi_* \cE$ is a vector bundle, where $\pi \colon A \times B \to A$ is the first projection. Then the holomorphic fibrewise section space is exactly $\bV(\pi_* \cE)$. Let us notice however that the holomorphic mapping spaces are more naturally topologised using the analytic topology of the Hom scheme. Fortunately, Douady shows in \cite{douady_probleme_1966} that the inclusion inside the whole mapping space is continuous.
\end{remark}

To lighten the notation, and as no confusion can arise, we will from now on drop the symbol $\bV(-)$ when considering continuous sections of a vector bundle.

\begin{definition}
Taking fibrewise, over $\Pic(X)$, continuous global sections of $J^1_p\cP$ which are never vanishing, we obtain the space
\begin{align*}
    \Gammaconfib(J^1_p\cP \setminus 0) &:= \Gammaconfib(J^1_p\cP \setminus 0 \to \Pic(X)) \\
    &= \left\{ ([\cL], s) \mid [\cL] \in \Pic(X), \ s \in \Gammacon(J^1\PL \setminus 0) \right\}.
\end{align*}
The group $\bC^\times$ acts by multiplying the sections by scalars and we let $\Gammaconfib(J^1_p\cP \setminus 0) / \bC^\times$ be the quotient for that action.
\end{definition}

\subsection{The main theorem}

By \cref{remark:holomorphic-continuous-fibrewise-sections}, the fibrewise jet map followed by the inclusion of the space of holomorphic sections inside continuous sections gives rise to a continuous map
\[
    j^1 \colon \Mhyp^\alpha \lra \Gammaconfib(J^1_p\cP_\alpha \setminus 0) / \bC^\times.
\]
Denote by $\bP(J^1\cO_X)$ the projectivisation of the first jet bundle of $\cO_X$ on $X$. We will make use of the following:
\begin{proposition}[{Compare \cite[Lemma~2.5 and Section~3 p.~73]{crabb_function_1984}}]\label{proposition:radon-hurwitz-connected-components}
The connected components of the section space $\Gammacon(\bP(J^1\cO_X))$ are in one-to-one correspondence with $H^2(X;\bZ)$. For a given Chern class $\alpha$, the associated connected component $\Gammacon^\alpha(\bP(J^1\cO_X))$ consists of those sections $s$ such that the pullback $s^*\cO(1)$ of the tautological bundle has Chern class $\alpha$.
\end{proposition}

It follows from the proposition that if $\cL$ is a line bundle with Chern class $\alpha$, the quotient map
\[
    \Gammacon(J^1\cL \setminus 0) \lra \Gammacon(\bP(J^1\cL)) \cong \Gammacon(\bP(J^1\cO_X))
\]
has image inside the connected component $\Gammacon^\alpha(\bP(J^1\cO_X))$. Here we have used that $J^1\cL \cong J^1\cO_X \otimes \cL$ and the fact that the projectivisation of a vector bundle is invariant under tensoring with a line bundle. 

Now, let $\cL_0$ be a chosen line bundle with Chern class $\alpha$. By choosing an isomorphism of topological line bundles $\cL_0 \cong \PL$ for each $[\cL] \in \Pic^\alpha(X)$, we obtain a map
\begin{equation}\label{equation:quotient-map-fibrewise-sections-projective}
    \Gammaconfib(J^1_p\cP_\alpha \setminus 0) / \bC^\times \lra \Gammacon(\bP(J^1\cL_0)) \cong \Gammacon(\bP(J^1\cO_X))
\end{equation}
which factors through $\Gammacon^\alpha(\bP(J^1\cO_X))$. As any two choices of isomorphisms $\cL_0 \cong \PL$ differ by a non-zero constant, we see that the map is indeed uniquely well-defined and continuous. The following is our main result:

\begin{theorem}\label{theorem:main-theorem}
Let $X$ be a smooth projective complex variety. Let $\alpha \in \NS(X)$ be such that $\alpha - c_1(K_X)$ is ample. The jet map
\[
    j^1 \colon \Mhyp^\alpha \lra \Gammacon^\alpha(\bP(J^1\cO_X))
\]
induces an isomorphism in integral homology in the range of degrees $* < \frac{d(X,\alpha) - 3}{2}$. (See \cref{definition:constant-minimal-jet-ampleness}.)
\end{theorem}

\section{Proof of the main theorem}\label{section:proof-main-theorem}

The proof of the main theorem is executed in two steps. In \cref{subsection:homology-isomorphism}, we first prove:
\begin{proposition}\label{proposition:main-homology-isomorphism}
Let $X$ and $\alpha$ be as in \cref{theorem:main-theorem}. The jet map
\[
    j^1 \colon \Mhyp^\alpha \lra \Gammaconfib(J^1_p\cP_\alpha \setminus 0) / \bC^\times
\]
induces an isomorphism in integral homology in the range of degrees $* < \frac{d(X, \alpha) - 3}{2}$.
\end{proposition}
Then, in \cref{subsection:homotopy-type-fibrewise-sections}, we show the following:
\begin{proposition}\label{proposition:main-homotopy-equivalence}
The map defined in~\eqref{equation:quotient-map-fibrewise-sections-projective}
\[
    \Gammaconfib(J^1_p\cP_\alpha \setminus 0) / \bC^\times \lra \Gammacon^\alpha(\bP(J^1\cO_X))
\]
is a weak homotopy equivalence.
\end{proposition}

\subsection{The homology isomorphism}\label{subsection:homology-isomorphism}

The jet map fits in the following diagram where the top row is its restriction to a fibre above an $[\cL] \in \Pic^\alpha(X)$:
\begin{center}
\begin{tikzcd}
\Gammans(\PL) / \bC^\times \arrow[d] \arrow[r] & \Gammacon(J^1\PL \setminus 0) / \bC^\times \arrow[d]             \\
\Mhyp^\alpha \arrow[r] \arrow[d]               & \Gammaconfib(J^1_p\cP_\alpha \setminus 0) / \bC^\times \arrow[d] \\
\Pic^\alpha(X) \arrow[r, equal]  & \Pic^\alpha(X)                                                  
\end{tikzcd}
\end{center}
The uppermost map was studied in \cite{aumonier_h-principle_2022} where the following result was proved:
\begin{theorem}[{Compare \cite[Corollary~8.1]{aumonier_h-principle_2022}}]\label{theorem:mon-theoreme-h-principe}
Let $\cL$ be a $d$-jet ample line bundle on a smooth projective complex variety $X$. Then the jet map
\[
    \Gammans(\cL) / \bC^\times \lra \Gammacon(J^1\cL \setminus 0) / \bC^\times
\]
induces an isomorphism in homology in the range of degrees $* < \frac{d-1}{2}$.
\end{theorem}

Now, if both lower vertical maps were fibrations, a comparison of the associated Serre spectral sequences would prove \cref{proposition:main-homology-isomorphism}. This is indeed the case for the map on the right-hand side. The other map is only a microfibration, which turns out to be sufficient for the argument to go through. We start by reviewing this technical notion popularised by Weiss in \cite{weiss_what_2005}.

\begin{definition}
A map $\pi \colon E \to B$ is called a Serre microfibration if for any $k \geq 0$ and any commutative diagram
\begin{center}
\begin{tikzcd}
\{0\} \times D^k \arrow[r, "u"] \arrow[d, hook] & E \arrow[d, "\pi"] \\
{[0,1] \times D^k} \arrow[r, "v"']              & B                 
\end{tikzcd}
\end{center}
there exists an $\varepsilon > 0$ and a map $h \colon [0,\varepsilon] \times D^k \to E$ such that $h(0,x) = u(x)$ and $\pi \circ h(t,x) = v(t,x)$ for all $x \in D^k$ and $t \in [0, \varepsilon]$.
\end{definition}

\begin{remark}
Any Serre fibration is a microfibration. More generally, the restriction of a Serre fibration to an open subspace of the total space is a microfibration.
\end{remark}

Contrary to the case of fibrations, the homotopy types of the fibres of a microfibration can vary. Nonetheless, we have the very useful comparison theorem of Raptis generalising a result of Weiss:
\begin{theorem}[{Compare \cite[Theorem~1.3]{raptis_serre_2017}}]\label{theorem:raptis-microfibration-comparison}
Let $p \colon E \to B$ be a Serre microfibration, $q \colon V \to B$ be a Serre fibration, and $f \colon E \to V$ a map over $B$. Suppose that $f_b \colon p^{-1}(b) \to q^{-1}(b)$ is $n$-connected for some $n \geq 1$ and for all $b \in B$. Then the map $f \colon E \to V$ is $n$-connected.
\end{theorem}

In the present situation, we only have access to \cref{theorem:mon-theoreme-h-principe} which provides an isomorphism in homology, rather than on homotopy groups. The remedy chosen here is to suspend the spaces as to obtain simply connected spaces and then apply the homology Whitehead theorem.

\begin{definition}
For a map $p \colon E \to B$, its fibrewise (unreduced) $k$\textsuperscript{th} suspension is defined to be
\[
	\Sigma^k_B E = \left( E \times [0,1] \times S^{k-1} \right) \bigg/ \big((e,0,s) \sim (e,0,s') \text{ and } (e,1,s) \sim (e',1,s) \text{ if } p(e) = p(e') \big).
\]
The fibre of the natural map $\Sigma^k_B p \colon \Sigma^k_B E \to B$ induced by $p$ is the unreduced $k$\textsuperscript{th} suspension of the fibre of $p$ (here modelled as the join with the sphere $S^{k-1}$):
\[
	(\Sigma^k_B p)^{-1}(b) = \Sigma^k p^{-1}(b), \quad \forall b \in B.
\]
\end{definition}

\begin{lemma}\label{lemma:projection-fibrewise-sections-pic-fibre-bundle}
The map 
\[
    \Gammaconfib(J^1_p\cP_\alpha \setminus 0) / \bC^\times \lra \Pic^\alpha(X)
\]
is a fibre bundle.
\end{lemma}
\begin{proof}
Let $U \subset \Pic^\alpha(X)$ be a small contractible open subset. A topological vector bundle being trivial over a contractible base, we obtain an isomorphism of vector bundles
\[
    \psi \colon J^1_p \cP_\alpha|_{U \times X} \overset{\cong}{\lra} U \times J^1 \cP_{[\cL_0]}
\]
over $U \times X$, with $[\cL_0] \in U$ a chosen basepoint. The map
\begin{align*}
    \left( \Gammaconfib(J^1_p\cP_\alpha \setminus 0) / \bC^\times \right)|_U &\lra U \times \Gammacon(J^1 \cP_{[\cL_0]} \setminus 0) \\
    ([\cL], s) &\longmapsto ([\cL], \psi \circ s)
\end{align*}
is then a homeomorphism over $U$ exhibiting the local triviality of the fibre bundle.
\end{proof}

We will say that a map $A \to B$ is homology $m$-connected if it induces an isomorphism on homology groups $H_i(A) \to H_i(B)$ for $i < m$ and a surjection when $i = m$.
\begin{lemma}\label{lemma:microfibration-comparison-after-suspension}
Let $q \colon V \to B$ be a fibre bundle, and $p \colon U \to B$ be the restriction of a fibre bundle $E \to B$ to an open subset $U \subset E$. Let $f \colon U \to V$ be a map over $B$ and suppose that for every $b \in B$, the restriction to the fibre
\[
	f_b \colon p^{-1}(b) \lra q^{-1}(b)
\]
is homology $m$-connected. Then $f \colon U \to V$ is homology $m$-connected.
\end{lemma}

\begin{proof}
For any $b \in B$, the suspension of $f$ on the fibre
\[
	\Sigma^2 f_b \colon \Sigma^2 p^{-1}(b) \lra \Sigma^2 q^{-1}(b)
\]
induces an isomorphism in homology in degrees $* \leq m+1$ and a surjective morphism in degree $* = m+2$. As both spaces are simply connected, the homology Whitehead theorem implies that this map is $(m+2)$-connected. We would like to apply \cref{theorem:raptis-microfibration-comparison} to $\Sigma^2_B f$, but $\Sigma^2_B U \subset \Sigma^2_B E$ is not open and it is unclear if $\Sigma^2_B U \to B$ is a microfibration. We resolve the issue by enlarging slightly the space to a homotopy equivalent one. More precisely, let
\[
	W = \big(\Sigma^2_B U\big) \cup \big( (E \times (0.5,1] \times S^1)) / \sim\big) \subset \Sigma^2_B E,
\]
and denote by $E_b, W_b, U_b$ the fibres of the respective spaces above a point $b \in B$. Using in each fibre the homotopy equivalence $\big((E_b \times (0.5,1] \times S^1)/\sim\big) \simeq S^1$ given by collapsing gives a homotopy equivalence
\[
	(W, W_b) \overset{\simeq}{\lra} (\Sigma^2_B U, \Sigma^2 U_b)
\]
for all $b \in B$. Now, the fibrewise suspension of the fibre bundles $E \to B$ and $V \to B$ are fibre bundles. As $W \subset \Sigma^2_B E$ is open, the restriction $W \to B$ is a microfibration. Applying \cref{theorem:raptis-microfibration-comparison} to the composite
\[
	W \overset{\simeq}{\lra} \Sigma^2_B U \overset{\Sigma^2_B f}{\lra} \Sigma^2_B V
\]
and using that the first map is a homotopy equivalence, we obtain that $\Sigma^2_B f \colon \Sigma^2_B U \to \Sigma^2_B V$ is $(m+2)$-connected. Hence it is homology $(m+2)$-connected. Comparing the Mayer--Vietoris sequences of the fibrewise suspensions finally shows that $f \colon U \to B$ is homology $m$-connected.
\end{proof}

\begin{proof}[Proof of \cref{proposition:main-homology-isomorphism}]
By \cref{definition:moduli-of-hypersurfaces}, the map $\Mhyp^\alpha \to \Pic^\alpha(X)$ is the restriction of the projective bundle $\bP(p_*\cP_\alpha) \to \Pic^\alpha(X)$ to the open subset $\Mhyp^\alpha$. Using \cref{theorem:mon-theoreme-h-principe} and \cref{lemma:projection-fibrewise-sections-pic-fibre-bundle}, we can apply \cref{lemma:microfibration-comparison-after-suspension} to conclude.
\end{proof}

\subsection{The homotopy type of the space of fibrewise sections}\label{subsection:homotopy-type-fibrewise-sections}

To prove that the map of \cref{proposition:main-homotopy-equivalence} is weak equivalence, we will extend it to a morphism between fibrations with equal total spaces, and then compare the homotopy fibres. (See~\eqref{equation:morphism-of-fibrations-hofibres} below for the precise diagram.) For this latter part, we shall need to work with convenient point-set models for the homotopy fibres, and thus we begin by making explicit some basic results in algebraic topology. The material of this section is very standard, but we include it to fix the notation.

For a pointed space $(A,a)$, we let $P(A,a) = \Map_*(([0,1],0), (A,a))$ be the space of paths in $A$ starting at $a$. We will write $\text{cte}_*$ for the constant loop based at a point $*$.

\subsubsection{The homotopy fibre of a homotopy fibre}
Let $\pi \colon (E,e_0) \to (B,b_0)$ be a fibration between pointed spaces, $F = \pi^{-1}(b_0)$ be the fibre, and $\Omega_{b_0} B$ be the loop space of $B$ based at $b_0$. Writing $i \colon F \hookrightarrow E$ for the inclusion, the space
\[
    Hi := \{ (e, \alpha) \in F \times P(E,e_0) \mid \alpha(1) = i(e) = e \}
\]
is a model of its homotopy fibre. It is well known that $Hi$ and $\Omega_{b_0} B$ are homotopy equivalent, and the goal of this small section is to recall an explicit description of the induced bijection on connected components.

\bigskip

We write
\[
    H\pi := \{ (e, \gamma) \in E \times P(B,b_0) \mid \gamma(1) = \pi(e) \}
\]
for the homotopy fibre of $\pi$. Let $j \colon H\pi \to E$ be the map $(e, \gamma) \mapsto e$, whose homotopy fibre is given by
\begin{align*}
    Hj &:= \{ (e, \gamma, \alpha) \in H\pi \times P(E,e_0) \mid \alpha(1) = e \} \\
    &\cong \{ (\gamma, \alpha) \in P(B,b_0) \times P(E,e_0) \mid \gamma(1) = \pi \circ \alpha(1) \}.
\end{align*}
The map $\Omega_{b_0} B \to Hj$ given by $\gamma \longmapsto (\gamma, \text{cte}_{e_0})$ is a homotopy equivalence (see \cite[Note~4.7.1]{dieck_algebraic_2008}). The situation is summarised in the following diagram:
\begin{center}
\begin{tikzcd}
Hi \arrow[r] \arrow[d, "\simeq"'] & F \arrow[d, "\simeq"', hook] \arrow[rd, "i", hook] &                     &   \\
Hj \arrow[r]                      & H\pi \arrow[r, "j"']                               & E \arrow[r, "\pi"'] & B \\
\Omega_{b_0}B \arrow[u, "\simeq"] &                                                    &                     &  
\end{tikzcd}
\end{center}

\begin{lemma}\label{lemma:bijection-pi0-lifting-paths}
Let $\gamma \colon [0,1] \to B$ be a loop based at $b_0$. Let $\alpha \colon [0,1] \to E$ be a lift of that loop starting at $e_0$. The map on connected components
\begin{align*}
    \pi_0(\Omega_{b_0} B) &\lra \pi_0(Hi) \\
    [\gamma] &\longmapsto [(\alpha(1), \alpha)]
\end{align*}
is well-defined and is a bijection.
\end{lemma}
\begin{proof}
The natural map $F \to H\pi$ is a homotopy equivalence as $\pi$ is a fibration. Hence the induced map $Hi \to Hj$ given by $(e,\alpha) \mapsto (\text{cte}_{b_0}, \alpha)$ is a homotopy equivalence. Therefore it suffices to show that $(\gamma, \text{cte}_{e_0})$ and $(\alpha(1), \alpha)$ are in the same connected component of $Hj$. Both are in the same component as $(\gamma, \alpha)$, as seen by deforming either the first or the second path.
\end{proof}

\subsubsection{Homotopy fibre of a principal bundle}
In this subsection, $\pi \colon (E,e_0) \to (B,b_0)$ is now a principal $G = \pi^{-1}(b_0)$-bundle. We let $\alpha \colon [0,1] \to B$ be a path from $b_0$ to a point $b_1$, and we choose a point $e_1 \in \pi^{-1}(b_1)$. As before, recall that a model for the homotopy fibre of $\pi$ is given by
\[
    H\pi := \{ (e, \gamma) \in E \times P(B,b_0) \mid \gamma(1) = \pi(e) \}.
\]
We may choose a lift of the path $\alpha$ to a path $\beta \colon [0,1] \to E$ such that $\pi \circ \beta = \alpha$. We define $e_1 = \beta(1)$. As the action of $G$ on $\pi^{-1}(b_1)$ is free and transitive, there exists a unique $g_1 \in G$ such that $g_1 \cdot e_1 = e_1'$ (where $\cdot$ denotes the action).
\begin{lemma}\label{lemma:connected-components-fibre-principal-bundle}
We keep the notation as above. Then the points $(e_1', \alpha)$ and $(g_1 \cdot e_0, \text{cte}_{e_0})$ are in the same connected component of the homotopy fibre $H\pi$.
\end{lemma}
\begin{proof}
The map $g_1 \cdot \beta \colon [0,1] \to E$ is a path from $g_1 \cdot e_0$ to $g_1 \cdot e_1 = e_1'$, and is such that $\pi \circ (g_1 \cdot \beta) = \pi \circ \beta = \alpha$. Thus the map
\begin{align*}
    [0,1] &\lra H\pi \\
    t &\longmapsto \big( (g_1 \cdot \beta)(t), \ \alpha(t \cdot -) \big)
\end{align*}
is a path from $(g_1 \cdot e_0, \text{cte}_{e_0})$ to $(e_1', \alpha)$ in $H\pi$.
\end{proof}

\subsubsection{The proof of \texorpdfstring{\cref{proposition:main-homotopy-equivalence}}{the homotopy equivalence}}
For concreteness, we start by fixing basepoints. Let $[\cL_0] \in \Pic^\alpha(X)$, and let $s_0 \in \Gammacon(J^1\cP_{[\cL_0]} \setminus 0)$. We will use these as basepoints, as well as the images $[s_0] \in  \Gammacon(J^1\cP_{[\cL_0]} \setminus 0) / \bC^\times$ and $\bP s_0 \in \Gammacon^\alpha(\bP(J^1\cO_X))$.

Pointwise multiplication of maps gives $\Map(X, \bC^\times)$ the structure of a topological group. By \cite[Proposition~2.6]{crabb_function_1984}, there is a principal $\Map(X,\bC^\times)$-bundle:
\begin{equation}
    \Map(X,\bC^\times) \lra \Gammacon(J^1\cP_{[\cL_0]} \setminus 0) \lra \Gammacon^\alpha(\bP(J^1\cO_X)).
\end{equation}
There is also the subgroup $\bC^\times \subset \Map(X, \bC^\times)$ of the constant functions, and modding out fibrewise gives a principal bundle:
\begin{equation}\label{equation:principal-mapping-bundle}
    \Map(X,\bC^\times) / \bC^\times \lra \Gammacon(J^1\cP_{[\cL_0]} \setminus 0) / \bC^\times \lra \Gammacon^\alpha(\bP(J^1\cO_X)).
\end{equation}
We obtain a commutative diagram of pointed spaces where each row is a fibration sequence
\begin{equation}\label{equation:morphism-of-fibrations-hofibres}
\begin{tikzcd}
F_1 \arrow[r] \arrow[d] & {\left( \Gammacon(J^1\cP_{[\cL_0]} \setminus 0)/\bC^\times, \ [s_0] \right)} \arrow[d, equal] \arrow[r] & {\left( \Gammaconfib(J^1_p \cP_\alpha \setminus 0) / \bC^\times, \ [s_0] \right)} \arrow[d] \\
F_2 \arrow[r]           & {\left( \Gammacon(J^1\cP_{[\cL_0]} \setminus 0)/\bC^\times, \ [s_0] \right)} \arrow[r]                                & {\left( \Gammacon^\alpha(\bP(J^1\cO_X)), \ \bP s_0 \right)}                                
\end{tikzcd}
\end{equation}
and the spaces $F_1$ and $F_2$ are defined as the respective homotopy fibres, using the models recalled in the previous section. Using the 5-lemma and the long exact sequence of homotopy groups associated to a fibration, \cref{proposition:main-homotopy-equivalence} follows directly from the next lemma.
\begin{lemma}
Using the notation as above, the map induced on the homotopy fibres $F_1 \to F_2$ is a homotopy equivalence.
\end{lemma}
\begin{proof}
We already know that $F_1 \simeq \Omega_{[\cL_0]} \Pic^\alpha(X)$ (cf. \cref{lemma:projection-fibrewise-sections-pic-fibre-bundle}) and $F_2 \simeq \Map(X, \bC^\times) / \bC^\times$ (cf. \cref{equation:principal-mapping-bundle}), which are both homotopy equivalent to the discrete space $H^1(X;\bZ)$. Therefore we only need to verify that the map $F_1 \to F_2$ induces a bijection on the set of connected components. We have a diagram of sets
\begin{center}
\begin{tikzcd}
{\pi_0 \big( \Omega_{[\cL_0]} \Pic^\alpha(X) \big)} \arrow[r, "\cong"', "1"{circled, yshift=2pt}] \arrow[d, dashed] & \pi_0(F_1) \arrow[d, "2"{circled, xshift=3pt}] \\
{\pi_0 \big( \Map(X,\bC^\times) / \bC^\times \big)} \arrow[r, "\cong", "3"{circled, swap, yshift=-2pt}]                    & \pi_0(F_2)          
\end{tikzcd}
\end{center}
where the right vertical map is induced from $F_1 \to F_2$, the top map \circlednum{1} is explained in \cref{lemma:bijection-pi0-lifting-paths}, the bottom map \circlednum{3} is induced by the inclusion of the fibre inside the homotopy fibre, and the dotted arrow is defined by composition. It suffices to show that this last arrow is a bijection.

To do so, we go through the composition and use the explicit descriptions of the maps recalled in the last section. Let $\gamma \in \Omega_{[\cL_0]} \Pic^\alpha(X)$ be a loop. Using the fibre bundle of \cref{lemma:projection-fibrewise-sections-pic-fibre-bundle}
\[
    \Gammaconfib(J^1_p\cP_\alpha \setminus 0) / \bC^\times \lra \Pic^\alpha(X)
\]
we choose a lift to a path
\[
    \alpha \colon [0,1] \to \Gammaconfib(J^1_p \cP_\alpha \setminus 0) / \bC^\times
\]
starting at $[s_0]$ and ending at some $[s_1']$. The map \circlednum{1} is given according to \cref{lemma:bijection-pi0-lifting-paths} by
\[
    \pi_0 \big( \Omega_{[\cL_0]} \Pic^\alpha(X) \big) \lra \pi_0(F_1), \quad [\gamma] \longmapsto [([s_1'], \alpha)].
\]
Write $\bP \alpha$ for the image of $\alpha$ in $\Gammacon^\alpha(\bP(J^1\cO_X))$. Applying \circlednum{2} then gives
\[
    [([s_1'], \alpha)] \longmapsto [([s_1'], \bP\alpha)].
\]
Using the fibration~\eqref{equation:principal-mapping-bundle}
\[
    \Gammacon(J^1\cP_{[\cL_0]} \setminus 0)/\bC^\times \lra \Gammacon^\alpha(\bP(J^1\cO_X))
\]
we lift the path $\bP\alpha$ to a path $\beta$ in $\Gammacon(J^1\cP_{[\cL_0]} \setminus 0)/\bC^\times$ starting at $[s_0]$ and ending at some point $[s_1]$. Using the principal bundle structure of~\eqref{equation:principal-mapping-bundle}, there is a unique class of a map $[\varphi_1] \in \Map(X,\bC^\times) / \bC^\times$ such that $[\varphi_1 \cdot s_1] = [s_1']$. By \cref{lemma:connected-components-fibre-principal-bundle}
\[
    [([s_1'], \bP\alpha)] = [([\varphi_1 \cdot s_1], \bP \alpha)] = [([\varphi_1 \cdot s_0], \bP \text{cte}_{s_0})].
\]
From the homotopy equivalence given by the inclusion of the strict fibre inside the homotopy fibre of the bundle~\eqref{equation:principal-mapping-bundle}:
\[
    \Map(X, \bC^\times) / \bC^\times \overset{\simeq}{\lra} F_2, \quad [\varphi] \longmapsto [([\varphi \cdot s_0], \bP \text{cte}_{s_0})]
\]
we see that the inverse of the map \circlednum{3} then sends
\[
    [([\varphi_1 \cdot s_0], \bP \text{cte}_{s_0})] \longmapsto [\varphi_1].
\]
We see first that it is surjective: indeed the composition of \circlednum{2} and the inverse of \circlednum{3} is surjective as any $[\varphi]$ is the image of $[([\varphi \cdot s_0], \text{cte}_{[s_0]})]$. Secondly we check that it is compatible with the group structures: on the source given by composition of loops, and on the target given by multiplication of maps. This is enough to finish the proof as both groups are isomorphic to $H^1(X;\bZ)$, so that any surjective morphism is fact an isomorphism.

We check the compatibility with the group structures using the notations as above. Let $\gamma_1, \gamma_2 \in \Omega_{[\cL_0]} \Pic^\alpha(X)$ be two loops. As above, choose lifts $\alpha_1$ and $\alpha_2$ starting at $[s_0]$ and ending at $[s_1']$ and $[s_2']$ respectively. Then lift $\bP\alpha_1$ and $\bP\alpha_2$ to paths $\beta_1, \beta_2$ starting at $[s_0]$ and ending at $[s_1]$ and $[s_2]$ respectively. Write $[\varphi_1\cdot s_1] = [s_1']$ and $[\varphi_2 \cdot s_2] = [s_2']$, for unique $[\varphi_1], [\varphi_2]$. As we have shown above, the composition of \circlednum{1}, \circlednum{2} and the inverse of \circlednum{3} maps
\[
    [\gamma_1] \longmapsto [\varphi_1], \quad [\gamma_2] \longmapsto [\varphi_2].
\]
Let us now consider the concatenation of loops $\gamma_1 \ast \gamma_2$. First, observe that there exists a unique $[\varphi]$ such that $[\varphi \cdot s_0] = [s_1']$. Therefore the loop $\gamma_1 \ast \gamma_2$ can be lifted to the path $\alpha_3 := \alpha_1 \ast (\varphi \cdot \alpha_2)$ ending at $[s_3'] := [\varphi \cdot s_2']$. Then we may lift $\bP\alpha_3 := \bP(\alpha_1 \ast (\varphi\cdot\alpha_2)) = \bP(\alpha_1 \ast \alpha_2)$ to the path $\beta_1 \ast (\varphi_1^{-1}\cdot\varphi\cdot\beta_2)$ ending at $[s_3]$. We write $[\varphi_3\cdot s_3] = [s_3']$ for a unique $[\varphi_3]$. Once again, the composition of \circlednum{1}, \circlednum{2} and the inverse of \circlednum{3} maps
\[
    [\gamma_1 \ast \gamma_2] \longmapsto [\varphi_3].
\]
Now observe that, as $[s_3]$ is the end point of $\beta_1 \ast (\varphi_1^{-1}\cdot\varphi\cdot\beta_2)$, we have
\[
    [s_3] = [\varphi_1^{-1}\cdot\varphi\cdot s_2].
\]
Therefore
\[
    [\varphi_3\cdot s_3] = [\varphi_3\cdot\varphi_1^{-1}\cdot\varphi\cdot s_2].
\]
But $[\varphi_3\cdot s_3] = [s_3'] = [\varphi \cdot s_2'] = [\varphi\cdot\varphi_2\cdot s_2]$. By uniqueness
\[
    [\varphi_3\cdot\varphi_1^{-1}\cdot\varphi] = [\varphi\cdot\varphi_2],
\]
hence $[\varphi_3] = [\varphi_1 \cdot \varphi_2]$, which achieves the proof.
\end{proof}

\section{Rational computations and stability}\label{section:rational-homotopy}

In this part, we show how \cref{theorem:main-theorem} can be used to make explicit computations of the rational cohomology of $\Mhyp^\alpha$. Assuming that the underlying variety $X$ is topologically parallelisable, we will also exhibit a phenomenon of homological stability.

We will first recall a general strategy, dating back to Haefliger \cite{haefliger_rational_1982}, to compute the cohomology of continuous section spaces. In \cref{theorem:rational-cohomology-section-space-projective-bundle} below, we provide a commutative differential graded algebra (CDGA) computing the rational cohomology of the section space of the projective bundle. We hope that this will convince the reader that the homotopical approach taken in this paper may be useful in practical computations. We will freely use the notations and results from rational homotopy theory. A textbook account can be found in \cite{felix_rational_2001}. In particular, we write $\Lambda(-)$ for the free commutative graded algebra. We let $n$ be the complex dimension of $X$.

\subsection{Haefliger's tower of section spaces}

Although \cref{theorem:main-theorem} provides an integral homology isomorphism, we will mainly be interested in the rational cohomology groups for computational reasons. Fibrewise rationalisation (denoted $(-)_{f\bQ}$) yields a fibration
\begin{equation}\label{equation:fibrewise-rationalised-bundle}
    \bP^n_\bQ \lra \bP(J^1\cO_X)_{f\bQ} \lra X.
\end{equation}
By \cite[Theorem~5.3]{moller_nilpotent_1987}, the natural map $\bP(J^1\cO_X) \to \bP(J^1\cO_X)_{f\bQ}$ induces a map on section spaces
\[
    \Gammacon(\bP(J^1\cO_X)) \lra \Gammacon(\bP(J^1\cO_X)_{f\bQ})
\]
which is a rationalisation when restricted to a connected component on the source and target. (Beware the fact that the source has $H^2(X;\bZ)$ many connected components, while the target has $H^2(X;\bQ)$ many of them.) We apply the general strategy described in \cite[Section~1.3]{haefliger_rational_1982} to compute the rational homotopy type of the section space $\Gammacon(\bP(J^1\cO_X)_{f\bQ})$. The fibration~\eqref{equation:fibrewise-rationalised-bundle} admits a Moore--Postnikov decomposition of the form
\begin{center}
\begin{tikzcd}
{K(\bQ,2n+1)} \arrow[r] & Y_2 \simeq \bP(J^1\cO_X)_{f\bQ} \arrow[d, "p_2"] &               \\
{K(\bQ,2)} \arrow[r]    & Y_1 \arrow[d, "p_1"] \arrow[r, "k_1"]            & {K(\bQ,2n+2)} \\
                        & Y_0 = X \arrow[r, "k_0"]                         & {K(\bQ,3)}   
\end{tikzcd}
\end{center}
where each $p_i \colon Y_i \to Y_{i-1}$ is a principal fibration classified by the $k$-invariant $k_{i-1}$. The latter were computed by Møller:
\begin{lemma}[{Compare \cite[Lemma~2.1]{moller_homology_1985}}]\label{lemma:moller-k-invariants}
The $k$-invariant $k_0$ is trivial. In particular $Y_1 \simeq X \times K(\bQ,2)$. Writing $z \in H^2(K(\bQ,2);\bQ)$ for the generator, $k_1$ corresponds to the cohomology class
\[
    \pushQED{\qed}
    \sum_{i=0}^{n+1} (-1)^i c_i(J^1\cO_X) \otimes z^{n+1-i} \in H^*(X;\bQ) \otimes H^*(K(\bQ,2);\bQ). \qedhere
    \popQED
\]
\end{lemma}

Let $s \in \Gammacon^\alpha(\bP(J^1\cO_X)_{f\bQ})$ with $\alpha \in H^2(X;\bQ)$. The map $p_2 \circ s$ is a section of $Y_1 \to X$, and we denote by $\Gamma_1 \subset \Gammacon(Y_1 \to X)$ its connected component. As $k_0$ is trivial by \cref{lemma:moller-k-invariants}, there is a homotopy equivalence
\[
    \Gammacon(Y_1 \to X) \simeq \Map(X, K(\bQ,2)) \simeq K(\bQ,2) \times K(H^1(X;\bQ),1) \times H^2(X;\bQ).
\]
and $\Gamma_1$ corresponds to the connected component indexed by $\alpha$.

\begin{lemma}[{Compare \cite[Lemma~2.2]{moller_homology_1985}}]\label{lemma:moller-pullback-generator}
Let $\Psi$ be the composite
\begin{equation}\label{equation:composite-evaluation-k-invariant}
    \Psi \colon K(\bQ,2) \times K(H^1(X;\bQ),1) \times X \simeq \Gamma_1 \times X \overset{\mathrm{ev}}{\lra} Y_1 \overset{k_1}{\lra} K(\bQ,2n+2).
\end{equation}
Let $z \in H^2(K(\bQ,2);\bQ)$ be the generator. Let $\{x_j\}$ be a basis of $H^1(X;\bZ)$ and let $\{x_j'\}$ be the dual basis of $H^1(K(H^1(X;\bZ),1);\bZ) \cong H^1(X;\bZ)^\vee$. The morphism induced in cohomology $\Psi^*$ sends the generator $\chi \in H^{2n+2}(K(\bQ,2n+2);\bQ)$ to the class:
\[
    \pushQED{\qed}
    \Psi^*(\chi) = \sum_{i=0}^{n+1} (-1)^i \bigg(1\otimes 1 \otimes c_i(J^1\cO_X)\bigg) \cup \bigg(z \otimes 1 \otimes 1 + 1 \otimes 1 \otimes \alpha + \sum_j 1 \otimes x_j' \otimes x_j\bigg)^{n+1-i}. \qedhere
    \popQED
\]
\end{lemma}

Let $\overline{k_1} \colon \Gamma_1 \to K(\bQ,2n+2)^X$ be the adjoint of the map~\eqref{equation:composite-evaluation-k-invariant}. There is a homotopy equivalence (see \cite[Section~1]{haefliger_rational_1982})
\begin{equation}\label{equation:mapping-space-to-EM-product}
    K(\bQ,2n+2)^X \simeq \prod_{i=2}^{2n+2} K(H^{2n+2-i}(X;\bQ), i).
\end{equation}
\begin{lemma}[Compare \cite{haefliger_rational_1982}]\label{lemma:haefliger-adjoint-map-in-cohomology}
Let $\varphi_i$ be the map to the $i$-th factor of the product:
\[
    \varphi_i \colon \Gamma_1 \overset{\overline{k_1}}{\lra} K(\bQ,2n+2)^X \lra K(H^{2n+2-i}(X;\bQ),i).
\]
The morphism induced in cohomology is given explicitly by:
\begin{align*}
    \varphi_i^* \colon H^{2n+2-i}(X;\bQ)^\vee \cong H^i(K(H^{2n+2-i}(X;\bQ);\bQ) &\lra H^i(\Gamma_1;\bQ) \\
    y' &\longmapsto y' \cap \Psi^*(\chi).
\end{align*}
Here, for $w \otimes y \in H^*(\Gamma_1) \otimes H^*(X)$ and $y' \in H^*(X)^\vee$, we write $y' \cap (w \otimes y) = y'(y) w$. \qed
\end{lemma}

\begin{proposition}[Compare \cite{haefliger_rational_1982}]\label{proposition:haefliger-fibration-of-section-spaces}
There is a fibration
\[
    K(\bQ,2n+1)^X \lra \Gammacon^\alpha(\bP(J^1\cO_X)_{f\bQ}) \lra \Gamma_1
\]
pulled back from the path space fibration over $K(\bQ,2n+2)^X$ via the map $\overline{k_1}$. \qed
\end{proposition}

\begin{theorem}\label{theorem:rational-cohomology-section-space-projective-bundle}
Let $z$ and $\{x_j'\}$ be as in \cref{lemma:moller-pullback-generator}. Let $\{y_{ik}'\}$ be a basis of the rational cohomology of~\eqref{equation:mapping-space-to-EM-product} where $y_{ik}' \in H^{2n+2-i}(X;\bQ)^\vee$ is in degree $i$. The rational cohomology of $\Gammacon^\alpha(\bP(J^1\cO_X)_{f\bQ})$ is given by the cohomology of the following commutative differential graded algebra:
\[
    \Lambda \big(z, \ x_j', \ s^{-1}y_{ik}' \big), \quad d(z) = 0, \ d(x_j') = 0, \ d(s^{-1}y_{ik}') = \varphi_i^*(y_{ik}')
\]
where $z$ is in degree $2$, each $x_j'$ is in degree $1$, each $s^{-1}y_{ik}'$ is in degree $i-1$, and $\varphi_i^*$ is given as in \cref{lemma:haefliger-adjoint-map-in-cohomology}.
\end{theorem}
\begin{proof}
By \cref{proposition:haefliger-fibration-of-section-spaces}, there is a homotopy pullback square:
\begin{center}
\begin{tikzcd}
\Gammacon^\alpha(\bP(J^1\cO_X)_{f\bQ}) \arrow[d] \arrow[r] & * \arrow[d]     \\
\Gamma_1 \arrow[r, "\overline{k_1}"]                       & {K(\bQ,2n+2)^X}
\end{tikzcd}
\end{center}
By the Eilenberg--Moore theorem, the cohomology of the pullback is given by the derived tensor product
\[
    \Lambda(z,x_j') \otimes^\mathbb{L}_{\Lambda(y_{ik}')} \bQ
\]
which can be computed by choosing $\Lambda(y_{ik}') \to \big(\Lambda(s^{-1}y_{ik}', y_{ik}'), \ d(s^{-1}y_{ik}') = y_{ik}' \big) \simeq \bQ$ as a cofibrant replacement.
\end{proof}

\begin{example}\label{example:torus-rational-cdga}
Let $X$ be a smooth curve ($n=1$) of genus $1$ (i.e. a torus). It is a framed manifold, hence its jet bundle has trivial Chern classes. Write $a, b$ for the standard basis of $H^1(X;\bZ)$ such that $a^2 = b^2 = 0$ and $u = ab$ generates $H^2(X;\bZ)$. Let $a',b'$ be the dual basis. Let $\alpha = k \cdot u$ for some $k \in \bZ$. With the notations of \cref{lemma:moller-pullback-generator} we have
\begin{align*}
    \Psi^*(\chi) &= \big(z \otimes 1 \otimes 1 + 1 \otimes 1 \otimes \alpha + 1 \otimes a' \otimes a + 1 \otimes b \otimes b'\big)^2 \\
        &= (2k(z \otimes 1) - 2(1 \otimes a'b')) \otimes u + 2(z \otimes a' \otimes a) + 2(z \otimes b' \otimes b) + z^2 \otimes 1 \otimes 1.
\end{align*}
The morphisms $\varphi_i^*$ of \cref{lemma:haefliger-adjoint-map-in-cohomology} are given by
\begin{align*}
    \varphi_2^* \colon u' &\longmapsto u' \cap \Psi^*(\chi) = 2k(z \otimes 1) - 2(1 \otimes a'b') \\
    \varphi_3^* \colon a' &\longmapsto a' \cap \Psi^*(\chi) = 2(z \otimes a') \\
        b' &\longmapsto b' \cap \Psi^*(\chi) = 2(z \otimes b') \\
    \varphi_4^* \colon 1 &\longmapsto 1 \cap \Psi^*(\chi) = z^2 \otimes 1.
\end{align*}
Therefore the cohomology of $\Gammacon^\alpha(\bP(J^1\cO_X)) \simeq \Map_\alpha(X, \bP^2)$ is given by the cohomology of the CDGA:
\[
    \Lambda(z, a', b', y_1, y_2, y_2', y_3), \quad d(y_1) = 2kz - 2a'b', \ d(y_2) = 2za', \ d(y_2') = 2zb', \ d(y_3) = z^2
\]
where the indices on the last four variables indicate their degrees. (See \cite[Section~3]{moller_homology_1985} for related computations.)
\end{example}

\subsection{Homological stability}

Despite the formula given in \cref{theorem:rational-cohomology-section-space-projective-bundle}, it is unclear to us how the cohomology varies when $\alpha$ does. Nonetheless, when $X$ is topologically parallelisable, we can make the following qualitative remark:

\begin{proposition}
Let $X$ be a smooth projective complex variety such that $\Omega^1_X$ is a topologically trivial vector bundle, and let $\alpha \in H^2(X;\bQ)$. Then there is a homotopy equivalence
\[
    \Gammacon^{\alpha}(\bP(J^1\cO_X)_{f\bQ}) \simeq \Gammacon^{k\alpha}(\bP(J^1\cO_X)_{f\bQ})
\]
for any non-zero rational number $k \in \bQ^\times$.
\end{proposition}
\begin{proof}
As $X$ is topologically parallelisable, the jet bundle $J^1\cO_X$ is topologically trivial. Hence the section space is the mapping space into the fibre:
\[
    \Gammacon^{k\alpha}(\bP(J^1\cO_X)_{f\bQ}) \simeq \Map_{k\alpha}(X, \bP^n_\bQ)
\]
where the subscript $k\alpha$ on the right-hand side indicates the connected component of the maps which pullback the generator in cohomology to $k\alpha$. Post-composing with a self map of $\bP^n_\bQ$ of degree $1/k$ gives a homotopy equivalence
\[
    \Map_{k\alpha}(X, \bP^n_\bQ) \simeq \Map_{\alpha}(X, \bP^n_\bQ). \qedhere
\]
\end{proof}

\begin{corollary}\label{corollary:rational-stability}
Let $X$ be a smooth projective complex variety which is topologically parallelisable. Let $\alpha \in \NS(X)$ be ample, and assume that $d(X,\alpha) \geq 1$ (see \cref{definition:constant-minimal-jet-ampleness}). Then, for any integer $k \geq 1$, there is map
\[
    \Mhyp^{k\alpha} \lra \Map_{\alpha}(X,\bP^n_\bQ)
\]
inducing an isomorphism in rational cohomology in the range of degrees $* < \frac{k \cdot d(X,\alpha) - 3}{2}$. In particular, the rational cohomology stabilises as $k \to \infty$. \qed
\end{corollary}

\begin{remark}
The rational homotopy type of the mapping space $\Map(X,\bP^n_\bQ)$ can be easily computed without the results of the last section. In \cite[Theorem~1.4]{berglund_rational_2015}, Berglund gives an explicit $L_\infty$-algebra model whose underlying graded $\bQ$-vector space is given by
\[
    H^*(X;\bQ) \otimes \bQ\cdot\{u,w\}
\]
where $H^i(X)$ sits in degree $-i$, and $u,w$ are respectively in degrees $1$ and $2n$. (This uses that $X$ is a formal space.) For the connected component corresponding to $\alpha \in H^2(X;\bQ)$, the associated Maurer--Cartan element is $\tau = \alpha \otimes u$. In particular, the only possibly non-vanishing brackets are given by
\[
    [x_1 \otimes u, \ldots, x_r \otimes u]_\tau = \pm \frac{(n+1)!}{(n+1-r)!} (\alpha^{n+1-r} x_1 x_2\cdots x_r) \otimes w.
\]
In fact, in the case of the torus, the Chevalley--Eilenberg complex associated to this $L_\infty$-algebra is the CGDA given in \cref{example:torus-rational-cdga}.
\end{remark}

\section{Scanning and configuration spaces on curves}\label{section:scanning-configuration-spaces}

In this section, we explain how the present article fits into the general philosophy of scanning maps in topology. In \cref{theorem:scanning-for-curves}, we recover a special case of a result of McDuff about the homology of configuration spaces of points on a curve.

\subsection{Scanning}

We begin with a brief and intuitive exposition of the general idea behind scanning. Suppose given $M \subset N$, a $d$-dimensional submanifold of an $n$-dimensional manifold. We can try to see what $M$ looks like by looking locally at each point of $N$. One can imagine looking through a magnifying glass: either we are far from $M$ and see nothing, or close to $M$ and see a first-order approximation of $M$, i.e. a tangent space, together with a small vector from the center of the lens to $M$. To formalise this intuition, recall the tautological quotient bundle over the Grassmannian of $d$-dimensional planes in $\bR^n$:
\[
    \bR^{n-d} \lra \gamma_{d,n}^\perp := \{ (H,v) \mid H \in \Grass(d, \bR^n), \; v \in \bR^n / H \} \lra \Grass(d, \bR^n).
\]
One thinks of a point $(H,v) \in \gamma_{d,n}^\perp$ as a $d$-dimensional plane together with a normal vector. The Thom space $\Grass(d, \bR^n)^{\gamma_{d,n}^\perp}$ is obtained by one-point compactifying the total space. This construction can be done fibrewise to the tangent bundle $TN$ of $N$, and we denote by $\Grass(d,TN)^{\gamma_{d,n}^\perp}$ the resulting bundle over $N$. The submanifold $M$ then gives a section 
\[
    N \lra \Grass(d,TN)^{\gamma_{d,n}^\perp}
\]
obtained by sending a point far away from $M$ to the point at infinity (in the Thom space), and sending a point $x \in N$ close to a point $y \in M$ to the tangent space $T_yM \subset T_yN$ together with the vector pointing from $y$ to $x$. Of course this requires to be made precise, e.g. by choosing a tubular neighbourhood of $M$ inside $N$. In many cases, this idea can be implemented in families to obtain a map from a parameter space of submanifolds to a section space.

\bigskip

We are now ready to give an interpretation of the jet map of \cref{theorem:main-theorem} in the spirit of scanning. We take $X$ and $\alpha$ as in the assumptions of that theorem. We also denote by $n = \dim_\bC X$ the complex dimension of $X$. The main observation is the following:
\begin{lemma}
For integers $d, m$, let $\Grass(d, \bC^m)$ be the Grassmannian of complex $d$-dimensional planes in $\bC^m$ and $\gamma_{d,m}^\perp$ be the tautological quotient bundle. There is a homeomorphism
\begin{align*}
    \gamma_{d,m}^\perp &\overset{\cong}{\lra} \Grass(d+1, \bC^m \oplus \bC) \setminus \Grass(d+1, \bC^m) \\
    (H,v) &\longmapsto (H,0) \oplus (v,1)
\end{align*}
where $\Grass(d+1, \bC^m)$ is embedded inside $\Grass(d+1, \bC^m \oplus \bC)$ via $P \mapsto (P,0)$. \qed
\end{lemma}
When $V$ is an $n$-dimensional complex vector space, the tautological quotient bundle over $\Grass(n-1,V) = \bP(V)$ is homeomorphic to $\bP(V \oplus \bC) \setminus \{*\}$. Hence its Thom space is $\bP(V \oplus \bC)$. From the isomorphism $J^1\cO_X \cong \Omega^1_X \oplus \cO_X$ as smooth complex vector bundles, we see that
\[
    \bP(J^1\cO_X) \cong \Grass(n-1, \Omega^1_X)^{\gamma_{n-1,n}^\perp} \cong \Grass(n-1, TX)^{\gamma_{n-1,n}^\perp}.
\]
Under these identifications, the jet map
\[
    \Mhyp^\alpha \lra \Gammacon \big( \Grass(n-1, TX)^{\gamma_{n-1,n}^\perp} \big)
\]
is very close to the general idea of scanning described above. Given a hypersurface $V(s) \subset X$, the derivative $x \mapsto ds(x)$ records the tangent space when non-zero, i.e. near the hypersurface, and $x \mapsto s(x)$ records in some sense the distance to the hypersurface, an analogue of the normal vector. 

\subsection{Configuration spaces on curves}

Let us now describe the case $n=1$ in more details. The variety $X$ is then a curve and we think of $\alpha \in \NS(X)$ as an integer under the isomorphism $\NS(X) \subset H^2(X;\bZ) \cong \bZ$ given by the complex orientation. A hypersurface of Chern class $\alpha$ is simply an unordered configuration of $\alpha$ points and we have a homeomorphism
\begin{align*}
    \Mhyp^\alpha &\overset{\cong}{\lra} \UConf_\alpha(X) \\
    ([\cL] \in \Pic^\alpha(X), \ [s] \in \Gammans(\cP_{[\cL]})/\bC^\times) &\lra V([s]).
\end{align*}
There is also an identification
\[
    \bP(J^1\cO_X) \cong \Grass(0, TX)^{\gamma_{0,1}^\perp} \cong \dot{TX}
\]
with the fibrewise one-point compactification of the tangent bundle. In \cite{mcduff_configuration_1975}, McDuff studied a scanning map on configuration spaces of points on a manifold, i.e. spaces of $0$-dimensional submanifolds. In the present work, we instead study (complex) codimension $1$ submanifolds. On a curve these agree and we recover a special case of McDuff's theorem, although our scanning map is now more algebraic in nature:
\begin{theorem}\label{theorem:scanning-for-curves}
Let $X$ be a smooth projective complex curve of genus $g$. Let $\alpha \in H^2(X;\bZ) \cong \bZ$ be such that $\alpha > 2g-2$. The jet map
\[
    \UConf_\alpha(X) \cong \Mhyp^\alpha \lra \Gammacon^\alpha(\bP(J^1\cO_X)) \cong \Gammacon^\alpha(\dot{TX})
\]
induces an isomorphism in integral homology in the range of degrees $* < \frac{\alpha - 2g - 3}{2}$.
\end{theorem}

\begin{proof}
This is a direct consequence of \cref{theorem:main-theorem}. To verify the assumption on the ampleness of $\alpha - c_1(K_X)$, recall that the canonical divisor has degree $2g -2$ and that a line bundle of positive degree is ample. The final bound is obtained by computing $d(X, \alpha) = \alpha - 2g$ using Riemann--Roch as explained in \cref{lemma:riemann-roch-jet-ampleness}.
\end{proof}

\begin{remark}
In \cite{mcduff_configuration_1975}, McDuff does not provide any concrete range for the homology isomorphism. However, the range $* \leq \alpha/2$ was later proved by Segal in \cite[Appendix~A]{segal_topology_1979}. In particular, it is better than the one given in the present article.
\end{remark}

\begin{remark}
The observations above lead us to ask about subvarieties of greater codimension: could the spaces $\Grass(n-c, \Omega^1_X)^{\gamma_{n-c,n}^\perp}$ be related to Hilbert schemes of codimension $c$ smooth subvarieties?
\end{remark}

\section{Characteristic classes and manifold bundles}\label{section:characteristic-classes}

In this section, we comment on the stable rational cohomology of $\Gammans(\cL) / \bC^\times$. Our main motivation is trying to relate it to the stable cohomology of moduli spaces of manifolds as investigated by Galatius and Randal-Williams \cite{galatius_stable_2014,galatius_homological_2017,galatius_homological_2018}. None of this section uses the new results of this article, and will in fact be deduced entirely from \cite{aumonier_h-principle_2022}. Nonetheless, we think that its fits naturally with the ``moduli space point of view" adopted in this paper.

We will assume that $\dim_\bC X = n \geq 4$ and that the fundamental group of $X$ is virtually polycyclic (eg. trivial) to apply the results of \cite{galatius_moduli_2019,friedrich_homological_2017}. We also choose a very ample line bundle $\cL$ on $X$.

\subsection{Recollections on stable classes}

We will shortly recall from \cite{aumonier_h-principle_2022} the geometric interpretation of the stable classes in the rational cohomology of $\Gammans(\cL)$. As we are here mostly interested in the quotient by the scalars $\bC^\times$, we first make two observations.
\begin{lemma}\label{lemma:discriminant-is-hypersurface}
Let $\cL$ be a $2$-jet ample line bundle on $X$. Then there exists a homogeneous polynomial
\[
    \Delta \colon \Gamma(\cL) \lra \bC
\]
such that $\Gammans(\cL) = \Delta^{-1}(\bC^\times)$. In other words, the complement of $\Gammans / \bC^\times \subset \bP(\Gamma(\cL))$ is a hypersurface given by the vanishing of the \emph{discriminant} polynomial $\Delta$.
\end{lemma}
\begin{proof}
The general theory of discriminants is explained in \cite{gelfand_discriminants_1994}, see eg. page 15 for the definition of the discriminant polynomial. In general, the subspace of singular sections of $\cL$ has codimension at least $1$, and exactly $1$ is most cases, see eg. \cite[Corollary~1.2]{gelfand_discriminants_1994}. We show that when $\cL$ is $2$-jet ample, we are indeed in the latter case. Consider the incidence variety
\[
    R := \left\{ (x,H) \in X \times \bP(\Gamma(\cL)) \mid x \in H \text{ is a singular point}\right\}
\]
and its two projections $p_1 \colon R \to X$ and $p_2 \colon R \to \bP(\Gamma(\cL))$. The set of singular hypersurfaces is $p_2(R)$. By $1$-jet ampleness of $\cL$, $p_1$ is a vector bundle. This implies that $R$ is irreducible of dimension $\dim \bP(\Gamma(\cL)) - 1$. In \cite[Proposition~3.4]{katz_pinceaux_1973}, Katz shows that the locus of $(x,H) \in R$ where $x$ is a non-degenerate singular point of $H$ is open. By $2$-jet ampleness, it is non empty, hence dense as $R$ is irreducible. Now \cite[Proposition~3.5]{katz_pinceaux_1973} shows that $p_2$ is birational. Therefore $p_2(R)$ has codimension $1$.
\end{proof}
We learned the following from \cite[Lemma~2.7]{das_space_2021}:
\begin{lemma}\label{lemma:discriminant-leray-hirsch}
Suppose that $\cL$ is $2$-jet ample. Then there is an isomorphism of $H^*(\Gammans(\cL); \bQ)$-modules:
\[
    H^*(\Gammans(\cL); \bQ) \cong H^*(\Gammans(\cL)/ \bC^\times; \bQ) \otimes H^*(\bC^\times; \bQ).
\]
\end{lemma}
\begin{proof}
Let $\Delta \colon \Gamma(\cL) \to \bC$ be the discriminant provided by \cref{lemma:discriminant-is-hypersurface} so that $\Gammans(\cL) = \Gamma(\cL) \setminus \Delta^{-1}(0)$. There is a fibre bundle
\[
    \bC^\times \lra \Gammans(\cL) \lra \Gammans(\cL) / \bC^\times
\]
and for any fibre the map $\bC^\times \hookrightarrow \Gammans(\cL) \overset{\Delta}{\to} \bC^\times$ is of degree $\deg(\Delta) \neq 0$, hence an isomorphism on rational cohomology. The lemma then follows by the Leray--Hirsch theorem.
\end{proof}

Consider the universal bundle of hypersurfaces:
\begin{equation*}
    V(s) \lra U(\cL) := \{ (s, x) \in \Gammans(\cL) \times X \mid s(x) = 0 \} \overset{\pi}{\lra} \Gammans(\cL).
\end{equation*}
At each point $(s,x) \in U(\cL)$, the derivative $ds(x)$ is non-zero, thus giving a map
\begin{equation*}
    j \colon U(\cL) \lra \Omega^1_X \otimes \cL \setminus 0.
\end{equation*}
For $\cL$ ample enough such that the Euler class of $\Omega^1_X \otimes \cL$ is non zero, the cohomology of the target of $j$ in degrees $* \geq 2n$ is
\[
    H^{* \geq 2n}(\Omega^1_X \otimes \cL \setminus 0; \bZ) \cong H^{* \geq 1}(X; \bZ) [2n-1],
\]
where $[2n-1]$ indicates a shift of degrees.
\begin{proposition}[{Compare \cite[Proposition~8.6]{aumonier_h-principle_2022}}]\label{proposition:stable-cohomology-fibrewise-integration}
Let $\cL$ be a $d$-jet ample line bundle on a smooth projective complex variety $X$. Suppose furthermore that the Euler class of $\Omega^1_X \otimes \cL$ is non-zero. Then the morphism given by pulling back along $j$ and integrating along the fibres of $\pi$
\[
    \pi_! \circ j^* \colon \Lambda(H^{* \geq 2n}(\Omega^1_X \otimes \cL \setminus 0; \bQ)) \lra H^*(\Gammans(\cL); \bQ) \lra H^*(\Gammans(\cL) /\bC^\times; \bQ)
\]
is an isomorphism in the range of degrees $* < \frac{d-1}{2}$. (The second arrow is obtained from \cref{lemma:discriminant-leray-hirsch} by projecting away from the second tensor factor.)
\end{proposition}
\begin{proof}
Let us write $g \colon \Gammans(\cL) \times X \to J^1\cL \setminus 0$ for the jet map and $p \colon \Gammans(\cL) \times X \to X$ for the projection. Then \cite[Proposition~8.6]{aumonier_h-principle_2022} shows that
\[
    p_! \circ g^* \colon \Lambda(H^{* \geq 2n+2}(J^1\cL \setminus 0; \bQ)) \lra H^*(\Gammans(\cL) /\bC^\times; \bQ)
\]
is an isomorphism in the range $* < \frac{d-1}{2}$. There is a pullback square
\[
\begin{tikzcd}
\Omega^1_X \otimes \cL \setminus 0 \arrow[r, "\iota", hook] & J^1\cL \setminus 0 \\
U(\cL) \arrow[r, "i", hook] \arrow[u, "j"] & \Gammans(\cL) \times X \arrow[u, "g"']
\end{tikzcd}
\]
where $\iota$ and $g$ are transverse, as one sees directly from a local computation using that $\cL$ is globally generated. Thus $i_! \circ j^* = g^* \circ \iota_!$. Furthermore, $\pi = p \circ i$ so that $\pi_! = p_! \circ i_!$. The proposition now follows from the fact that $\iota_!$ sends $H^{* \geq 2n}(\Omega^1_X \otimes \cL \setminus 0; \bQ)$ isomorphically onto $H^{* \geq 2n+2}(J^1\cL \setminus 0; \bQ)$.
\end{proof}

\subsection{Comparison with diffeomorphisms}\label{section:comparison-with-diffeomorphisms}

Let us fix a non-singular section $s \in \Gammans(\cL)$ and write $H := V(s)$ for the associated hypersurface. From the point of view of the Hilbert scheme developed in \cref{section:moduli-of-hypersurfaces}, the subspace $\Gammans(\cL) / \bC^\times$ classifies algebraic bundles with fibres equivalent to $H$ (as divisors) and embedded in $X$. Seeing $H$ only as a smooth oriented real manifold, one can consider the classifying space $\BDiff^\mathrm{or}(H)$ of its group of orientation preserving diffeomorphisms. By definition, this latter space classifies fibre bundles with fibre $H$ and structure group $\mathrm{Diff}^\mathrm{or}(H)$. In particular, there is a map
\begin{equation}\label{equation:classying-map-to-bdiff}
    \Gammans(\cL) / \bC^\times \lra \BDiff^\mathrm{or}(H)
\end{equation}
classifying the universal bundle
\begin{equation}\label{equation:universal-bundle-over-gammans-mod-c}
    U(\cL) / \bC^\times := \{ ([s], x) \in \Gammans(\cL) / \bC^\times \times X \mid s(x) = 0 \} \overset{\pi}{\lra} \Gammans(\cL) / \bC^\times.
\end{equation}
One could wonder if the map~\eqref{equation:classying-map-to-bdiff} induces an isomorphism in rational cohomology in a range of degrees. This was shown to be false when $X = \bP^n$ and $\cL = \cO(d)$ by Randal-Williams \cite{randal-williams_millermoritamumford_2019}. On the other hand, one could alter the situation by considering diffeomorphisms preserving other kinds of tangential structures: we have picked orientation, but could have chosen a spin structure in some cases, or a map to a background space, etc. We describe below a peculiar tangential structure $\theta$ such that the map classifying the universal bundle
\[
    \Gammans(\cL) \lra \BDiff^\theta(H)
\]
is ``very close" to being a rational homology isomorphism in a range of degrees.

\bigskip

Choose maps $[TX] \colon X \to BU(n)$ and $[\cL] \colon X \to BU(1)$ classifying respectively the tangent bundle of $X$ and $\cL$ as topological complex vector bundles. Let $B$ be the space defined by the following homotopy pullback square:
\begin{equation}\label{equation:homotopy-pullback-defining-tangential-structure}
\begin{tikzcd}
B \arrow[rr] \arrow[d]                                                       &  & X \arrow[d, "{\big( [TX], \ [\cL] \big)}"] \\
BU(n-1) \times BU(1) \arrow[rr, "{\big( ``\oplus", \ \mathrm{pr}_2 \big)}"'] &  & BU(n) \times BU(1)                        
\end{tikzcd}
\end{equation}
where $``\oplus" \colon BU(n-1) \times BU(1) \to BU(n)$ classifies taking the direct sum of vector bundles. We will adopt the point of view of spaces over $BO(2n-2)$ to describe tangential structures. (See \cite[Section~4.5]{galatius_moduli_2019} for a discussion.) In this language, the \emph{tangential structure} $\theta$ simply means the map:
\[
    \theta \colon B \lra BU(n-1) \times BU(1) \overset{\mathrm{pr}_1}{\lra} BU(n-1) \lra BO(2n-2).
\]

\begin{remark}
Informally, a $\theta$-structure on a $(2n-2)$-manifold $M$ is the datum of a lift of the map classifying the tangent bundle:
\begin{center}
\begin{tikzcd}
                                          & B \arrow[d, "\theta"] \\
M \arrow[r, "{[TM]}"'] \arrow[ru, dashed] & BO(2n-2)             
\end{tikzcd}
\end{center}
up to homotopy. By the universal property of the homotopy pullback, this amounts to providing two maps
\[
    M \lra BU(n-1) \times BU(1) \quad \text{ and } \quad M \lra X
\]
which become homotopic after further composing to $BU(n) \times BU(1)$ and such that $M \to BU(n-1)$ classifies the tangent bundles of $M$. In other words, this is the data of a map $\iota \colon M \to X$ and a complex line bundle $\cL'$ (corresponding to $M \to BU(1)$) such that $TH \oplus \cL' \cong \iota^*TX$ and $\iota^* \cL \cong \cL'$. Therefore, a $\theta$-structure on $M$ should be intuitively interpreted as a choice of an immersion $\iota \colon M \to X$ with normal bundle $\iota^*\cL$.
\end{remark}

We have chosen to construct $B$ via the homotopy pullback~\eqref{equation:homotopy-pullback-defining-tangential-structure} as it allowed us to informally understand the geometric meaning of a $\theta$-structure. But it turns out that we can give more familiar expressions for $B$ and the bundle classified by $\theta$ as explained in the following two lemmas.

\begin{lemma}\label{lemma:B-homotopy-equivalent-sphere-bundle}
There is a homotopy equivalence above $X$
\[
    B \simeq \Omega^1_X \otimes \cL \setminus 0.
\]
\end{lemma}
\begin{proof}
We will use the following explicit point set models:
\begin{align*}
    BU(j) &:= \{ P \subset \bC^\infty \mid \text{$P$ is a $j$-dimensional plane} \}, \\
    \gamma_j &:= \{ (P, v) \mid P \in BU(j), \ v \in P \}, \\
    \gamma_j^\vee &:= \{ (P, \varphi) \mid P \in BU(j), \ \varphi \colon P \to \bC \text{ linear map}\},
\end{align*}
for the classifying space, and its tautological vector bundle and the dual of it.
Recall that the classical fibration sequence
\[
    \bC^n \setminus 0 \lra BU(n-1) \lra BU(n)
\]
can be modelled by the sphere bundle of the dual tautological bundle using the homeomorphism
\[
    \gamma_n^\vee \setminus 0 \cong BU(n-1), \quad (P, \varphi \colon P \twoheadrightarrow \bC) \mapsto \ker(\varphi).
\]
From the pullback square
\begin{center}
\begin{tikzcd}
BU(n-1) \arrow[d] \arrow[r] \arrow[dr, phantom, "\scalebox{1.5}{$\lrcorner$}" , very near start, color=black]                                  & BU(n) \arrow[d]    \\
BU(n-1) \times BU(1) \arrow[r, "{(``\oplus", \mathrm{pr}_2)}"'] & BU(n) \times BU(1)
\end{tikzcd}
\end{center}
we see that the homotopy fibre of the bottom map is $\bC^n \setminus 0$. In fact, we can likewise model the fibre sequence
\[
    \bC^n \setminus 0 \lra BU(n-1) \times BU(1) \overset{(``\oplus", \mathrm{pr}_2)}{\lra} BU(n) \times BU(1)
\]
by the sphere bundle $\gamma_n^\vee \boxtimes \gamma_1 \setminus 0$. Indeed, we may write
\[
    \gamma_n^\vee \boxtimes \gamma_1 \setminus 0 \cong \{ (P, L, \varphi \otimes v) \mid P \in BU(n), \ L \in BU(1), \ \varphi \otimes v \in P^\vee \otimes L \setminus 0 \}
\]
and use the homeomorphism
\[
    \gamma_n^\vee \boxtimes \gamma_1 \setminus 0 \cong BU(n-1) \times BU(1), \quad (P, L, \varphi \otimes v) \mapsto (\ker(\varphi \otimes v), L).
\]
Under these identifications, one can check that the map $``\oplus"$ can be modelled by
\[
    (P, L, \varphi \otimes v) \longmapsto \ker(\varphi \otimes v) \oplus L \subset \bC^\infty \oplus \bC^\infty \cong \bC^\infty.
\]
Hence, by pulling back along $([TX], [\cL])$, one sees that $B \simeq \Omega^1_X \otimes \cL \setminus 0$.
\end{proof}

\begin{lemma}\label{lemma:virtual-vector-bundle-is-genuine}
Let $q \colon \Omega^1_X \otimes \cL \setminus 0 \to X$ be the projection. The virtual vector bundle $q^*(TX - \cL)$ is in fact the genuine vector bundle $\theta^*\gamma$ classified by the map $\theta \colon \Omega^1_X \otimes \cL \setminus 0 \to BO(2n-2)$.
\end{lemma}
\begin{proof}
We will use the homeomorphism
\[
    \Omega^1_X \otimes \cL \setminus 0 \cong \{ (x, \varphi) \mid x \in X, \ \varphi \colon TX|_x \twoheadrightarrow \cL|_x \text{ surjective linear map} \}
\]
given by identifying a non-zero vector of $(\Omega^1_X \otimes \cL)|_x$ with a surjective linear map. As one sees from the point set models described in the proof of \cref{lemma:B-homotopy-equivalent-sphere-bundle}, the pullback vector bundle $\theta^*\gamma$ classified by $\theta$ is equivalent to the one whose fibre above a point $(x, \varphi)$ is given by the kernel of $\varphi$. Writing out the vector bundles
\begin{align*}
    q^*TX &= \{ (x,\varphi,v) \mid (x,\varphi) \in \Omega^1\otimes \cL \setminus 0, \ v \in TX|_x \}, \\
    q^*\cL &= \{ (x,\varphi,v) \mid (x,\varphi) \in \Omega^1\otimes \cL \setminus 0, \ v \in \cL|_x \},
\end{align*}
we identify $\theta^*\gamma$ as the kernel of the morphism of vector bundles
\[
    q^*TX \lra q^*\cL, \ (x,\varphi,v) \longmapsto (x,\varphi,\varphi(v)).
\]
We thus obtain the short exact sequence of vector bundles
\[
    0 \lra \theta^*\gamma \lra q^*TX \lra q^*\cL \lra 0
\]
which proves the lemma.
\end{proof}

Let $H = V(s)$ be a smooth hypersurface with $s \in \Gammans(\cL)$. Using non-singularity, we obtain a map $\ell \colon H \to \Omega^1_X \otimes \cL \setminus 0$ given by $\ell(x) = ds(x)$.
\begin{lemma}\label{lemma:high-connectivity-lefschetz}
The map $\ell \colon H \to \Omega^1_X \otimes \cL \setminus 0$ is $(n-1)$-connected.
\end{lemma}
\begin{proof}
The inclusion $\iota \colon H \hookrightarrow X$ factors as
\[
    H \overset{\ell}{\lra} \Omega^1_X \otimes \cL \setminus 0 \lra X
\]
where the second map is the projection map of the bundle, hence $(2n-1)$-connected. Therefore it suffices to show that $\iota \colon H \to X$ is $(n-1)$-connected. But this is precisely the Lefschetz hyperplane theorem.
\end{proof}

Recall that a \emph{$\theta$-structure} on a manifold $M$ is a bundle map (fibrewise linear isomorphism) $TM \to \theta^*\gamma$, where $\gamma$ is the universal bundle above $BO(2n-2)$. In the proposition below, we observe that each $\ell \colon H \to \Omega^1_X \otimes \cL \setminus 0$ underlies a bundle map $\hat\ell \colon TH \to \theta^*\gamma$, and that these naturally assemble when varying $H$ in the universal bundle.
\begin{proposition}\label{proposition:bundle-of-theta-structures}
The universal bundle $U(\cL) \to \Gammans(\cL)$ admits the structure of a smooth fibre bundle with $\theta$-structure over $\ell$ in each fibre.
\end{proposition}
\begin{proof}
Let $T_vU(\cL)$ be the vertical tangent bundle of the universal bundle. By definition of a bundle with $\theta$-structure, we have to provide the horizontal maps in the following diagram to construct a vector bundle map
\begin{center}
\begin{tikzcd}
T_vU(\cL) \arrow[d] \arrow[r] & \theta^*\gamma \arrow[d]           \\
U(\cL) \arrow[r]              & \Omega^1_X \otimes \cL \setminus 0
\end{tikzcd}
\end{center}
which restricts to a linear isomorphism in each fibre.
Using the notation from the proof of \cref{lemma:virtual-vector-bundle-is-genuine}, we write
\[
    \Omega^1_X \otimes \cL \setminus 0 = \{ (x, \varphi) \mid x \in X, \ \varphi \colon TX|_x \twoheadrightarrow \cL|_x \text{ surjective linear map} \}
\]
and
\[
    \theta^*\gamma = \{ (x, \varphi, v) \mid (x,\varphi) \in \Omega^1_X \otimes \cL \setminus 0, \ v \in \ker(\varphi) \}.
\]
Differentiating a non-singular section $s \colon X \to \cL$ yields a short exact sequence of vector bundles
\[
    0 \lra TV(s) \lra TX \overset{ds}{\lra} \cL \lra 0
\]
which shows that $\ker(ds(x)) = TV(s)|_x$. In particular, we have a point-set model for the vertical tangent bundle given by
\[
    T_vU(\cL) = \left\{ (s,x,v) \mid s \in U(\cL), \ x \in V(s), \ v \in \ker(ds(x)) \right\}.
\]
Hence, taking
\[
    U(\cL) \lra \Omega^1_X \otimes \cL \setminus 0, \quad (s,x) \longmapsto (x, ds(x))
\]
and
\[
    T_vU(\cL) \lra \theta^*\gamma, \quad (s,x,v) \longmapsto (x, ds(x), v)
\]
gives the wanted square.
\end{proof}

Let us now look at hypersurfaces of higher degree. For every integer $d \geq 1$, we pick a section $s_d \in \Gammans(\cL^{\otimes d})$ and write $H_d = V(s_d) \subset X$ for the associated hypersurface. Replacing $\cL$ by $\cL^{\otimes d}$ in the diagram~\eqref{equation:homotopy-pullback-defining-tangential-structure}, we obtain spaces $B_d \simeq \Omega^1_X \otimes \cL^{\otimes d} \setminus 0$. We write $\theta_d \colon B_d \to BO(2n-2)$ for the tangential structure and $\hat\ell_d \colon TH_d \to \theta_d^*\gamma$ for the tangential structure on $H_d$ induced from the inclusion inside $X$ (as constructed in the proof of \cref{proposition:bundle-of-theta-structures}). Let $\mathrm{Bun}^{\theta_d}(H_d)$ denote the space of bundles maps $TH_d \to \theta_d^*\gamma$, and write
\[
    \cM^{\theta_d}(H_d) = \mathrm{Bun}^{\theta_d}(H_d) \sslash \mathrm{Diff}(H_d)
\]
for the Borel construction. As explained in \cite[Section~2.2]{galatius_moduli_2019}, this is a moduli space classifying smooth $H_d$-bundles with $\theta_d$-structure. We let $\cM^{\theta_d}(H_d, \hat\ell_d) \subset \cM^{\theta_d}(H_d)$ be the connected component of $\hat\ell_d$. Work of Galatius and Randal-Williams provides the following:
\begin{theorem}[{Compare \cite[Theorem~4.5]{galatius_moduli_2019}}]\label{theorem:oscar-soren-maintheorem}
Using the notations as above, let $\alpha = c_1(\cL)$ and $N := \int_X \alpha^{n+1} \neq 0$. There is a map
\[
    \cM^{\theta_d}(H_d, \hat\ell_d) \lra \Omega^\infty MT\theta_d \simeq \Omega^\infty \big(\Omega^1_X \otimes \cL^{\otimes d} \setminus 0 \big)^{q^*\cL^{\otimes d} - q^*TX}
\]
which, when regarded as a map onto the path component that it hits, induces an isomorphism in integral homology in degrees $* \leq \frac{1}{3}Cd^{n+1} + O(d^n)$, for some constant $C$ depending on $n$ and satisfying $\frac{1}{2}\cdot\frac{13}{15}N \leq C$.
\end{theorem}
\begin{proof}
The connectivity assumption of \cite[Theorem~4.5]{galatius_moduli_2019} is verified in \cref{lemma:high-connectivity-lefschetz}. The identification of the Thom spectrum is given by \cref{lemma:virtual-vector-bundle-is-genuine}. Finally, the range given is explained in \cite[Remark~5.6]{galatius_moduli_2019}
\end{proof}

A major advantage of that theorem resides in the fact that the rational homotopy of the infinite loop space is easily calculated:
\begin{lemma}
Suppose that the Euler class of $\Omega^1_X \otimes \cL^{\otimes d}$ does not vanish. Then there is an isomorphism of commutative rings
\begin{align*}
    H^*\big(\cM^{\theta_d}(H_d, \hat\ell_d); \bQ\big) &\cong \Lambda\big(H^{2n-1}(X)[1] \oplus H^1(X)[2] \oplus H^2(X)[3] \oplus \cdots \oplus H^{2n}(X)[2n+1]\big) \\
    &=: \Lambda\big(H^{2n-1}(X)[1]\big) \otimes \Lambda\big(H^{\bullet > 0}(X)[\bullet +1]\big)
\end{align*}
where $H^i(X)[j]$ denotes the graded $\bQ$-vector space $H^i(X;\bQ)$ placed in degree $j$.
\end{lemma}
\begin{proof}
This is a well-known computation in rational homotopy theory using the Thom isomorphism. See for example \cite[Remark~4.2]{galatius_moduli_2019}.
\end{proof}

By \cref{proposition:bundle-of-theta-structures}, the universal bundle is pulled back along a map
\begin{equation}\label{equation:classifying-map-to-theta-moduli}
    \Gammans(\cL^{\otimes d}) \lra \cM^{\theta_d}(H_d, \hat\ell_d).  
\end{equation}
Our work describes the stable rational cohomology of the domain, whereas Galatius and Randal-Williams compute the one of the codomain. We will shortly reveal the relation between these two rings of characteristic classes. First, as argued as the beginning of this section, it is more geometrically natural to consider the quotient of the domain of~\eqref{equation:classifying-map-to-theta-moduli} by the group $\bC^\times$. We have not yet seen a counterpart to this action on the codomain, so we explain how to proceed in the following:
\begin{proposition}\label{proposition:weird-cyclic-quotient}
Suppose that $\cL^{\otimes d}$ is $2$-jet ample, and denote by $\Delta \colon \Gamma(\cL^{\otimes d}) \to \bC$ the discriminant polynomial (see \cref{lemma:discriminant-is-hypersurface}). Let $\mu_{\deg \Delta} \subset \bC^\times$ be the cyclic subgroup of the ($\deg \Delta$)\textsuperscript{th} roots of unity. Then there is a homeomorphism
\[
    \Gammans(\cL^{\otimes d}) / \bC^\times \cong \Delta^{-1}(1) / \mu_{\deg \Delta},
\]
and an action of $\mu_{\deg \Delta}$ on $\cM^{\theta_d}(H_d, \hat\ell_d)$ such that the map~\eqref{equation:classifying-map-to-theta-moduli} fits into a commutative diagram:
\begin{center}
\begin{tikzcd}
\Gammans(\cL^{\otimes d}) \arrow[d] \arrow[r]    & {\cM^{\theta_d}(H_d, \hat\ell_d)} \arrow[d]           \\
\Gammans(\cL^{\otimes d}) / \bC^\times \arrow[r] & {\cM^{\theta_d}(H_d, \hat\ell_d) / \mu_{\deg \Delta}}
\end{tikzcd}
\end{center}
where both vertical arrows are the quotient maps for the actions and the bottom arrow is the induced map
\[
    \Gammans(\cL) / \bC^\times \cong \Delta^{-1}(1) / \mu_{\deg \Delta} \lra \cM^{\theta_d}(H_d, \hat\ell_d) / \mu_{\deg \Delta}.
\]
on the quotients from $\Delta^{-1}(1) \subset \Gammans(\cL^{\otimes d}) \to \cM^{\theta_d}(H_d, \hat\ell_d)$.
\end{proposition}
\begin{proof}
The classical theory of global Milnor fibrations provides a commutative diagram whose top two rows and leftmost two columns are fibre bundles:
\begin{center}
\begin{tikzcd}
\mu_{\deg \Delta} \arrow[d, hook] \arrow[r, hook] & \Delta^{-1}(1) \arrow[d, hook] \arrow[r]    & \Delta^{-1}(1) / \mu_{\deg \Delta} \arrow[d, "\cong"] \\
\bC^\times \arrow[d] \arrow[r, hook]              & \Gammans(\cL) \arrow[d, "\Delta"] \arrow[r] & \Gammans(\cL) / \bC^\times                            \\
\bC^\times / \mu_{\deg \Delta} \arrow[r, "\cong"] & \bC^\times                                  &                                                      
\end{tikzcd}
\end{center}
The inverse of the homeomorphism
\[
    \Delta^{-1}(1) / \mu_{\deg \Delta} \overset{\cong}{\lra} \Gammans(\cL)/\bC^\times
\]
is explicitly given by $s \mapsto \Delta(s)^{-1/ \deg \Delta} \cdot s$. This proves the first claim. For the second, recall that we have defined:
\[
    \cM^{\theta_d}(H_d, \hat\ell_d) = \mathrm{Bun}^{\theta_d}(H_d, \hat\ell_d) \sslash \mathrm{Diff}(H_d)
\]
where $\mathrm{Bun}^{\theta_d}(H_d, \hat\ell_d)$ denotes the connected component of $\hat\ell_d$ in the space of bundle maps. Recall also from the proof of \cref{proposition:bundle-of-theta-structures} the point-set model:
\[
    \theta^*\gamma = \{ (x, \varphi, v) \mid (x,\varphi) \in \Omega^1_X \otimes \cL \setminus 0, \ v \in \ker(\varphi) \}.
\]
The group $\bC^\times$ acts fibrewise on the vector bundle $\Omega^1_X \otimes \cL$, thus on $\theta^*\gamma$ via
\[
    \lambda \cdot (x, \varphi, v) = (x, \lambda \cdot \varphi, v) \quad \text{for } \lambda \in \bC^\times,
\]
and therefore acts on $\cM^{\theta_d}(H_d, \hat\ell_d)$ by post-composition on bundle maps. This action can be restricted to the cyclic subgroup $\mu_{\deg \Delta} \subset \bC^\times$. Finally the existence of the claimed commutative square follows by inspection.
\end{proof}

We are now ready to summarise the relation between the two rings of characteristic classes in the main result of this section:
\begin{theorem}\label{theorem:comparison-of-characteristic-classes}
Let $i \geq 0$ be an integer and let $d \gg 0$ be big enough so that
\[
    H^*(\cM^{\theta_d}(H_d, \hat\ell_d); \bQ) \cong \Lambda\big(H^{2n-1}(X)[1]\big) \otimes \Lambda\big(H^{\bullet > 0}(X)[\bullet +1]\big)
\]
and
\[
    H^*(\Gammans(\cL^{\otimes d}) / \bC^\times; \bQ) \cong \Lambda\big(H^{\bullet > 0}(X)[\bullet +1]\big)
\]
in degrees $* \leq i$. The map classifying the universal bundle
\[
    \Gammans(\cL^{\otimes d}) \lra \cM^{\theta_d}(H_d, \hat\ell_d)
\]
induces a ring morphism in rational cohomology with the following properties in cohomological degrees $* \leq i$:
\begin{enumerate}[(i)]
    \item Its restriction to $\Lambda\big(H^{\bullet > 0}(X)[\bullet +1]\big) \subset H^*(\cM^{\theta_d}(H_d, \hat\ell_d); \bQ)$ is injective.
    \item Its restriction to $\Lambda\big(H^{2n-1}(X)[1]\big) \subset H^*(\cM^{\theta_d}(H_d, \hat\ell_d); \bQ)$ is zero.
\end{enumerate}
In particular, its image in degrees $* \leq i$ is the subring $H^*(\Gammans(\cL)/ \bC^\times; \bQ) \subset H^*(\Gammans(\cL); \bQ)$, and the map
\[
    \Gammans(\cL)/ \bC^\times \lra \cM^{\theta_d}(H_d, \hat\ell_d) / \mu_{\deg \Delta}
\]
of \cref{proposition:weird-cyclic-quotient} induces an surjection in rational cohomology in degrees $* \leq i$.
\end{theorem}

\begin{proof}
Recall from \cite[Section~3.1]{galatius_moduli_2019} that the characteristic classes of $\theta_d$-bundles are given by integration along the fibres. From the similar description given in \cref{proposition:stable-cohomology-fibrewise-integration}, we see that, in degrees $* \leq i$, a basis of\[
    H^{\bullet > 0}(X)[\bullet +1] \subset H^*(\cM^{\theta_d}(H_d, \hat\ell_d); \bQ)
\]
is sent to a basis of
\[
    H^{\bullet > 0}(X)[\bullet +1] \subset H^*(\Gammans(\cL^{\otimes d}); \bQ)
\]
under the morphism induced by the map classifying the universal bundle. This proves the first point. To prove the second, we recall that the image of a element $w \in H^{2n-1}(X)$ is the fibre integration
\[
    \pi_!(i^*w) \in H^*(\Gammans(\cL^{\otimes d}); \bQ)
\]
where $\pi$ is the universal bundle, and $i \colon U(\cL^{\otimes d}) \to X \times \Gammans(\cL^{\otimes d}) \to X$ is the map $(f,x) \mapsto x$. From the commutative diagram
\begin{center}
\begin{tikzcd}
U(\cL^{\otimes d}) \arrow[d, "\pi"'] \arrow[r, "\iota"]           & \Gammans(\cL^{\otimes d}) \times X \arrow[d, "\mathrm{pr}_2"] \\
\Gammans(\cL^{\otimes d}) \arrow[r, equal] & \Gammans(\cL^{\otimes d}) 
\end{tikzcd}
\end{center}
we compute that
\[
    \pi_!(i^*w) = (\mathrm{pr}_2)_!(\iota_!(1) \cup 1 \otimes w) = (\mathrm{pr}_2)_!(1 \otimes c_1(\cL^{\otimes d}) \cup w) = 0.
\]
The identification of the subring corresponding to the image follows from \cref{lemma:discriminant-leray-hirsch}. Finally, to prove the last claim, it remains to show that the quotient map
\[
    \cM^{\theta_d}(H_d, \hat\ell_d) \lra \cM^{\theta_d}(H_d, \hat\ell_d) / \mu_{\deg \Delta}
\]
induces an isomorphism in rational cohomology. As the action is free, the quotient is the homotopy orbit space and we thus have a fibration:
\[
    \cM^{\theta_d}(H_d, \hat\ell_d) \lra \cM^{\theta_d}(H_d, \hat\ell_d) / \mu_{\deg \Delta} \lra B \mu_{\deg \Delta}.
\]
The monodromy action of $\pi_1(B \mu_{\deg \Delta}) = \mu_{\deg \Delta}$ on the cohomology of the fibre is trivial as it can be extended to the connected group $\bC^\times$. A finite group has trivial rational cohomology, here $H^*(B \mu_{\deg \Delta}; \bQ) = \bQ$, and the result follows.
\end{proof}

\begin{corollary}\label{corollary:simply-connected-gammans-bdiff}
Suppose that $H^{2n-1}(X;\bQ)$ vanishes (e.g. $X$ is simply connected) and let $i,d$ be as in \cref{theorem:comparison-of-characteristic-classes}. Then $\Gammans(\cL^{\otimes d}) / \bC^\times = \Mhyp^{d\alpha}$ for $\alpha = c_1(\cL)$ and the map
\[
    \Mhyp^{d\alpha} = \Gammans(\cL^{\otimes d}) / \bC^\times \lra \cM^{\theta_d}(H_d, \hat\ell_d) / \mu_{\deg \Delta}
\]
induce an isomorphism rational cohomology in degrees $* \leq i$. \qed
\end{corollary}

\begin{remark}
By general theory, a map from a space $T$ to the homotopy orbit space $\cM^{\theta}(H, \hat\ell) \mathbin{/\mkern-6mu/} \mu_{\deg \Delta}$ is given by a principal $\mu_{\deg \Delta}$-bundle $P \to T$ and an equivariant map $P \to \cM^{\theta}(H, \hat\ell)$. From that point of view, the map
\[
    \Gammans(\cL) / \bC^\times \lra \cM^{\theta}(H, \hat\ell) / \mu_{\deg \Delta}
\]
is given by the datum of a $\theta$-structure on the pullback of the universal bundle along the étale cover
\[
    \Delta^{-1}(1) \lra \Gammans(\cL)/\bC^\times
\]
with Galois group $\mu_{\deg \Delta}$.
\end{remark}

\subsubsection{Removing the quotient}

Using the general machinery of \cite[Section~4.3]{galatius_moduli_2019}, we construct a new tangential structure $\theta_d'$ which takes into account the $\mu_{\deg \Delta}$-action. The group $\mu_{\deg \Delta}$ acts on $B_d \simeq \Omega^1_X \otimes \cL^{\otimes d} \setminus 0$ through the scalar action on the bundle $\cL^{\otimes d}$, and we have:
\begin{lemma}
The map $\theta_d \colon B_d \to BO(2n-2)$ factors through the orbit space of the $\mu_{\deg \Delta}$-action via a map which we denote
\[
    \theta_d' \colon B_d / \mu_{\deg \Delta} \lra BO(2n-2).
\]
\end{lemma}
\begin{proof}
One can argue using the point-set models described in \cref{lemma:virtual-vector-bundle-is-genuine}. Indeed, recall that $B_d \simeq \Omega^1_X \otimes \cL^{\otimes d} \setminus 0$ can be identified with the space of pairs $(x,\varphi)$ of a point $x \in X$ and a linear surjection $\varphi \colon TX|_x \twoheadrightarrow \cL^{\otimes d}|_x$, and the map $\theta_d$ to $BO(2n-2)$ sends $(x,\varphi)$ to the kernel $\ker(\varphi)$. The group $\mu_{\deg \Delta}$ acts on that space by scalar multiplication on $\cL^{\otimes d}$, and this preserves the kernel of an epimorphism $\varphi$.
\end{proof}

The analogue of \cref{proposition:bundle-of-theta-structures} also holds for this new tangential structure:
\begin{lemma}\label{lemma:theta-prime-structure}
The universal bundle above $\Gammans(\cL^{\otimes d}) / \bC^\times$ admits the structure of a smooth fibre bundle with $\theta_d'$-structure.
\end{lemma}
\begin{proof}
A point of the universal bundle is a pair $([f], x)$ of an equivalence class $[f] \in \Gammans(\cL^{\otimes d}) / \bC^\times$ and a point $x \in V(f)$. We can choose a representative $f$ of the class $[f]$ with discriminant $\Delta(f) = 1$ and take its derivative at $x$ to define a map:
\[
    ([f], x) \longmapsto [df(x)] \in \big(\Omega^1_X \otimes \cL^{\otimes d} \setminus 0\big) / \mu_{\deg \Delta}.
\]
Although $df(x)$ is only well-defined up to a root of unity, its class is well-defined in the quotient. This map thus construct a $\theta_d'$-structure on the universal bundle.
\end{proof}

As before, if $H_d = V(f)$ is a smooth hypersurface with $f \in \Gammans(\cL^{\otimes d}) / \bC^\times$, we write $\hat\ell_d' \colon TH_d \to \theta_d'^*\gamma$ for the tangential structure constructed in the proof of \cref{lemma:theta-prime-structure} above. Finally we obtain:
\begin{corollary}\label{corollary:without-mu-quotient}
Suppose that $H^{2n-1}(X;\bQ) = 0$, and let $i,d$ be as in \cref{theorem:comparison-of-characteristic-classes}. The map classifying the universal bundle
\[
    \Mhyp^{c_1(\cL^{\otimes d})} = \Gammans(\cL^{\otimes d}) / \bC^\times \lra \cM^{\theta_d'}(H_d, \hat\ell_d')
\]
induces an isomorphism in rational cohomology in the range $* \leq i$.
\end{corollary}
\begin{proof}
The differential map $\ell_d' \colon H_d = V(f) \to (\Omega^1_X \otimes \cL^{\otimes d} \setminus 0) / \mu_{\deg \Delta}$ is not $(n-1)$-connected anymore (compare with \cref{lemma:high-connectivity-lefschetz}). But the composition
\[
    H_d \lra \Omega^1_X \otimes \cL^{\otimes d} \setminus 0 \lra (\Omega^1_X \otimes \cL^{\otimes d} \setminus 0) / \mu_{\deg \Delta}
\]
is a Moore--Postnikov $(n-1)$-stage. The group-like topological monoid of weak equivalences of $\Omega^1_X \otimes \cL^{\otimes d} \setminus 0$ above its quotient by $\mu_{\deg \Delta}$ is directly seen to be equivalent to
\[
    \Map(X, \mu_{\deg \Delta}) = \mu_{\deg \Delta}.
\]
Hence \cite[Theorem~4.5]{galatius_moduli_2019} applies (homotopy orbits and strict orbits are equivalent here as the group action is free) and shows that $\cM^{\theta_d'}(H_d, \hat\ell_d')$ and $\cM^{\theta_d}(H_d, \hat\ell_d) / \mu_{\deg \Delta}$ have the same stable rational cohomology. The result then follows by \cref{corollary:simply-connected-gammans-bdiff}.
\end{proof}

\subsection{En rød sild}
We finish this section by a remark on the infinite loop space from \cref{theorem:oscar-soren-maintheorem}. As explained in \cite[Theorem~8.11]{aumonier_h-principle_2022}, there is a map
\[
    \Gammans(\cL^{\otimes d}) \lra \Omega^{\infty+1} X^{J^1\cL^{\otimes d} - TX}
\]
which induces an isomorphism in rational cohomology in the range of degrees $* < \frac{d-1}{2}$. \footnote{To be precise, \cite[Theorem~8.11]{aumonier_h-principle_2022} only states that the map is $2n$-connected. However the connectivity solely appears as the connectivity of the map $S^{2n+1} \to \Omega^\infty \Sigma^\infty S^{2n+1}$. The latter is a rational equivalence, and the proof can be repeated after rationalisation.}
To compare this infinite loop space with the one appearing above, we will use the following well-known lemma:
\begin{lemma}
Let $V, W$ be two vector bundles on a space $Z$. We may assume that $W$ is a virtual vector bundle. Then there is a cofibre sequence of Thom spectra:
\[
    S(V)^W \lra Z^W \lra Z^{V \oplus W}.
\]
\end{lemma}
\begin{proof}
The Thom space $Z^V$ is defined by collapsing the sphere bundle $S(V)$ inside the disc bundle $D(V)$. Thus there is a cofibre sequence of spaces
\[
    S(V) \hookrightarrow D(V) \lra Z^V.
\]
The lemma follows by passing to Thom spectra with respect to the virtual bundle $W$ and using the equivalence $D(V) \simeq Z$.
\end{proof}

Applying the lemma to $Z = X$, $V = \Omega^1_X \otimes \cL^{\otimes d}$ and $W = \cL^{\otimes d} -TX$ yields the fibre sequence of spaces
\[
    \Omega^{\infty + 1}X^{J^1\cL^{\otimes d} - TX} \lra \Omega^\infty S(\Omega^1_X \otimes \cL^{\otimes d})^{\cL^{\otimes d} - TX} \lra \Omega^\infty X^{\cL^{\otimes d} - TX},
\]
in which we recognise the two infinite loop spaces appearing earlier in this section. Finally, let us remark that the rational homotopy type of the rightmost space
\[
    (\Omega^\infty X^{\cL^{\otimes d} - TX})_\bQ \simeq H^2(X;\bQ) \times K(H^1(X;\bQ), 1)
\]
is not far from that of $\Pic(X)$. This was one of the starting observations for the present article, although it now seems to me to be red herring.

\appendix

\section{Range estimates for jet ampleness}\label{section:range-estimates}

\subsection{The case of curves}

When $X$ is a curve, that is of dimension 1, one can give an explicit formula for the jet ampleness of a line bundle depending only on its degree.
\begin{lemma}\label{lemma:riemann-roch-jet-ampleness}
Let $X$ be a smooth projective complex curve of genus $g$. Let $\cL$ be a line bundle on $X$ and denote by $c_1(\cL) \in H^2(X;\bZ) \cong \bZ$ its degree, i.e. its first Chern class. Let $d \geq 1$ be an integer. If $c_1(\cL) > 2g-1+d$ then $\cL$ is $d$-jet ample.
\end{lemma}
\begin{proof}
As there is only tangent direction at each point on a curve, to show that $\cL$ is $d$-jet ample it suffices to show that the restriction map
\[
    H^0(\cL) \lra H^0(\cL \otimes \cO_Z)
\]
is surjective for all subschemes $Z \subset X$ of length $d+1$. Recall the short exact sequences of sheaves
\[
    0 \lra \cI_Z \lra \cO_X \lra \cO_Z \lra 0
\]
where $\cI_Z$ denotes the ideal sheaf of $Z$. From the long exact sequence in cohomology, we see that it suffices to show that
\[
    H^1(\cL \otimes \cI_Z) = 0.
\]
By Serre duality, this group is isomorphic to $H^0(K_X \otimes \cL^{-1} \otimes \cI_Z^{-1})$ where $K_X$ is the canonical sheaf. It is now enough to show that $K_X \otimes \cL^{-1} \otimes \cI_Z^{-1}$ has negative degree under the assumptions of the lemma. This follows by computing that $\deg \cI_Z^{-1} = d+1$, $\deg \cL^{-1} = -c_1(\cL)$ and $\deg K_X = 2g-2$ by Riemann--Roch.
\end{proof}

\subsection{The case of toric varieties}

When $X$ is a smooth projective toric variety, its fundamental group is trivial, hence the Picard scheme is discrete. In that case the results of this paper are simply obtained from \cite{aumonier_h-principle_2022}. We nevertheless comment on how to compute the jet ampleness of a line bundle to give a sense of the difficulty of the problem.

The basic idea is as follows: if $\cL$ is a $d$-jet ample line bundle on $X$, then so is its restriction to any rational curve on $X$. On such a rational curve $C \cong \bP^1$, a line bundle is of the form $\cO_{\bP^1}(a)$ and is $d$-jet ample if and only if $a \geq d$. Now there are some distinguished curves on $X$, namely the ones invariant under the torus action, and it turns out to be enough to check jet ampleness on them:
\begin{theorem}[{Compare \cite{rocco_generation_1999}}]
Let $\cL$ be a line bundle on a smooth projective toric variety. Then $\cL$ is d-jet ample if and only if $\cL \cdot C \geq d$ for any torus invariant curve $C \subset X$.
\end{theorem}
In \cite{rocco_generation_1999}, Di Rocco also proves two more equivalent criteria for jet ampleness in terms of convexity of the support function of $\cL$ and Seshadri constants at each point of $X$. We refer to that paper for the full details. Importantly for us, the criterion shows that $d$-jet ampleness can be checked by a finite number of inequations.

\subsection{Fujita's conjecture and jet ampleness on surfaces}

Whereas Kleiman's criterion shows that ampleness is a numerical property, jet ampleness, or even just very ampleness or global generation, is a trickier question to settle. In 1985, Fujita proposed the following conjecture which remains unsolved in general:
\begin{conjecture}[Compare \cite{fujita_polarized_1987}]
Let $X$ be a smooth projective complex variety of dimension $n$. Let $A$ be an ample line bundle on $X$. Then $K_X + (n+1)A$ is globally generated, and $K_X + (n+2)A$ is very ample.
\end{conjecture}

In dimension 1, the conjecture follows from the Riemann--Roch theorem. In higher dimension, the approach taken for curves would require proving a Kodaira-type vanishing theorem for non-invertible sheaves. However, in dimension 2, the conjecture was solved by Reider by different means:
\begin{theorem}[Compare \cite{reider_vector_1988}]
Fujita's conjecture is true for $n = 2$.
\end{theorem}
We recommend the lecture notes of Lazarsfeld \cite{lazarsfeld_lectures_1994} for a beautiful introduction to Fujita's conjecture and Reider's theorem.

\subsection{Cohomological criterion}

Although effective vanishing theorems like Riemann--Roch do not exist in higher dimension, there exist alternatives that can be used to provide qualitative statements about jet ampleness. The starting observation is the following cohomological criterion:
\begin{lemma}\label{lemma:cohomological-criterion}
Let $\cL$ be a line bundle on a smooth projective variety $X$ of dimension $n$. Let $d \geq 1$ be an integer. Then $\cL$ is $d$-jet ample if the cohomology groups
\[
    H^1(\cL \otimes \cI_Z)
\]
vanish for all 0-dimensional subschemes $Z \subset X$ of length $\sum\limits_{j=0}^d \binom{n+j-1}{j}$, and ideal sheaf $\cI_Z$.
\end{lemma}
\begin{proof}
By definition, $\cL$ is $d$-jet ample if the evaluation map
\[
    H^0(\cL) \lra H^0(\cL/ \fm_1^{k_1} \cdots \fm_l^{k_l}) \cong \bigoplus_{i=1}^l H^0(\cL / \fm_i^{k_i})
\]
is surjective for all distinct closed points $x_1,\ldots,x_l$ with associated maximal ideal sheaves $\fm_1,\ldots,\fm_l$, and all integers $k_i \geq 1$ such that $\sum k_i = d+1$. For a closed point with ideal sheaf $\fm$ and $k \geq 1$ an integer, the subscheme given by the ideal $\fm^k$ has length $\sum_{j=0}^{k-1} \binom{n+j-1}{j}$. Therefore the subscheme $Z$ given by the ideal sheaf $\fm_1^{k_1} \cdots \fm_l^{k_l}$ has length
\[
    l(Z) = \sum_{i = 1}^l \sum_{j = 0}^{k_i-1} \binom{n+j-1}{j} \leq \sum_{j=0}^d \binom{n+j-1}{j}.
\]
Now, if $Z$ is a subscheme with ideal sheaf $\cI_Z$, then surjectivity of $H^0(\cL) \to H^0(\cL \otimes \cO_Z)$ is implied by vanishing of the cohomology group $H^1(\cL \otimes \cI_Z)$ as one sees from the long exact sequence in cohomology associated to the short exact sequence of sheaves $0 \to \cI_Z \to \cO_X \to \cO_Z \to 0$.
\end{proof}

\begin{remark}
The results of \cite{aumonier_h-principle_2022} are stated in terms of the jet ampleness of $\cL$, which is why we use the same phrasing in this paper. But in fact, as we are only concerned with conditions on the first order derivatives of sections, we could settle for the following ad hoc weaker notion: a line bundle $\cL$ is $d$-good if the evaluation map
\[
    H^0(\cL) \lra \bigoplus_{i=1}^l H^0(\cL / \fm_i^2)
\]
is surjective for all distinct closed points $x_1,\ldots,x_l$ with associated maximal ideal sheaves $\fm_1,\ldots,\fm_l$, and $2l \leq d+1$. We claim that the proofs of \cite{aumonier_h-principle_2022} go through to study $\Gammans(\cL)$ with this weaker assumption. However, as the bounds we obtain to estimate $d$-goodness or jet ampleness are not very explicit, we have opted for the stronger assumption of jet ampleness which is more commonly studied.
\end{remark}

\begin{proposition}\label{proposition:arbitrarily-big-jet-ampleness}
Let $X$ be smooth projective complex variety, and $\alpha, \beta \in \NS(X)$ with $\alpha$ arbitrary and $\beta$ ample. Then for any integer $d \geq 1$, there exists an integer $k_0 \in \bN$ such that for all $k \geq k_0$ all line bundles of first Chern class equal to $\alpha + k\beta$ are $d$-jet ample.
\end{proposition}
\begin{proof}
Let us denote by $\cB$ an ample line bundle with first Chern class $\beta$. Let $M = \sum_{j=0}^d \binom{n+j-1}{j}$ and write $\Hilb_M(X)$ for the Hilbert scheme 0-dimensional subschemes of $X$ of length $M$. We will show that there exists a $k_0 \in \bN$ such that 
\[
    H^1(X, \cI_Z \otimes \cL \otimes \cB^{\otimes k}) = 0 \quad \forall k \geq k_0, \ \forall Z \in \Hilb_M(X), \ \forall \cL \in \Pic^\alpha(X).
\]
As tensoring with $\cB^{\otimes k}$ maps $\Pic^\alpha(X)$ isomorphically onto $\Pic^{\alpha +k\beta}(X)$, the proposition will follow by \cref{lemma:cohomological-criterion}. As $\Hilb_M(X) \times \Pic^\alpha(X)$ is proper (over $\Spec(\bC)$), we can find a finite covering
\[
    \Hilb_M(X) \times \Pic^\alpha(X) = \bigcup_{i=1}^N \Spec(A_i)
\]
by spectra of $\bC$-algebras of finite type. Let us write $X_{A_i} = X \times \Spec(A_i)$ and $p_i \colon X_{A_i} \to X$ for the first projection. Then $p_i^*\cB$ is ample by \stacks{0892}. Let $\cI$ be the universal family on $\Hilb_M(X) \times X$, let $\cP$ be a Poincar\'e sheaf on $\Pic^\alpha(X) \times X$, and denote by $\cI \otimes \cP$ the induced sheaf on $\Hilb_M(X) \times \Pic^\alpha(X) \times X$. Notice that it is a coherent sheaf that is flat over $\Hilb_M(X) \times \Pic^\alpha(X)$. Therefore its restriction to any $X_{A_i}$ is a coherent sheaf, flat over $\Spec(A_i)$. By Serre's coherent vanishing, see \stacks{0B5U}, there exists an $N_i \geq 0$ such that for all $k \geq N_i$ we have
\begin{equation}\label{equation:pushforward-cohomology}
    R^j (p_i)_*(\cI \otimes \cP \otimes p_i^*\cB^{\otimes k}) = 0 \quad \forall j > 0.
\end{equation}
We claim that this implies that
\begin{equation}\label{equation:fibrewise-cohomology}
    H^j(X, \cI_Z \otimes \cL \otimes \cB^{\otimes k}) = 0
\end{equation}
for all $j>0$, $k \geq N_i$, and $(Z, \cL) \in \Spec(A_i) \subset \Hilb_M(X) \times \Pic^\alpha(X)$. Assuming this claim, observe that taking $k_0 = \max N_i$ then proves the proposition. Fix such $k,Z,\cL$, we will prove\footnote{We have learned this argument from \cite[p. 29]{kleiman_picard_2005}.} the claim by downward induction on $j$. Let us write $y = (Z,\cL)$ for brevity. By coherent vanishing above dimension, \cref{equation:fibrewise-cohomology} is true for $j > \dim X$. Suppose we have vanishing for some $j \geq 2$. By \cite[7.5.3]{grothendieck_elements_1963}, \cref{equation:pushforward-cohomology} implies that 
\[
    R^j (p_i)_*(\cI \otimes \cP \otimes p_i^*\cB^{\otimes k} \otimes p_i^*\cF)_y = 0
\]
for any quasi-coherent sheaf $\cF$ on $\Spec(A_i)$. By the long exact sequence of derived functors, we see that the functor
\[
    \cF \longmapsto R^{j-1} (p_i)_*(\cI \otimes \cP \otimes p_i^*\cB^{\otimes k} \otimes p_i^*\cF)_y
\]
is right-exact. Under the correspondence between quasi-coherent sheaves on affine scheme and modules, this can be seen as a functor from $A_i$-modules to abelian groups. We may thus apply \cite[7.2.5]{grothendieck_elements_1963} to deduce a canonical isomorphism
\[
    R^{j-1} (p_i)_*(\cI \otimes \cP \otimes p_i^*\cB^{\otimes k})_y \otimes \cF_y \overset{\cong}{\lra} R^{j-1} (p_i)_*(\cI \otimes \cP \otimes p_i^*\cB^{\otimes k} \otimes p_i^*\cF)_y
\]
The left-hand side vanishes by virtue of \cref{equation:pushforward-cohomology}. Taking $\cF = \kappa(y)$ (the residue field), we obtain \cref{equation:fibrewise-cohomology} for $j-1$.
\end{proof}

\begin{remark}
For surfaces, there is a simpler and more explicit proof using Fujita's conjecture. Indeed, the image of the ample cone under the map $A \mapsto K_X + (n+2)A$ is a cone, and any line bundle having Chern class in that cone is very ample by Reider's theorem. If $\cL$ is a very ample line bundle, the component of the Picard scheme corresponding to $K_X + (n+2)A + (d-1)\cL$ only contains line bundles that are $d$-jet ample.
\end{remark}

\printbibliography

\end{document}